\documentclass[10pt,a4paper]{amsart}
\usepackage[T1]{fontenc}
\usepackage[english]{babel}
\usepackage{amsmath,amssymb,epsf,epic,eepic,pifont}
\usepackage{graphicx}
\usepackage{psfrag}
\usepackage{amsfonts}
\usepackage[all,knot]{xy}
\usepackage{epsfig}
\usepackage{bbm}
\usepackage{tabularx}
\usepackage{multirow}
\DeclareMathSizes{10}{9}{7}{7}
\xyoption{arc}
\newtheorem{theorem}{\bfseries Theorem}[section]
\newtheorem{lem}[theorem]{\bfseries Lemma}
\newtheorem{pro}[theorem]{\bfseries Proposition}

\theoremstyle{definition}

\newtheorem{cor}[theorem]{\bfseries Corollary}
\newtheorem*{thm1}{{\bfseries Theorem} \ref{thm:inv}}
\newtheorem*{thm2}{{\bfseries Theorem} \ref{HQFT}}
\newtheorem*{thm3}{{\bfseries Theorem} \ref{yetterHQFT}}
\newtheorem*{promax}{{\bfseries Proposition} \ref{pro:max}}
\newcommand{\tr}{\mathrm{tr}}

\newcommand{\Zz}{\mathbb{Z}}

\newcommand{\N}{\mathbb{N}}
\newcommand{\Cc}{\mathbb{C}}

\newcommand{\lc}{\Lambda_{\mathcal{C}}}

\newcommand{\C}{\mathcal{C}}

\newcommand{\pt}{\otimes}

\newcommand{\xd}[1]{#1^{\vee}}
\newcommand{\xg}[1]{{}^{\vee}#1}
\newcommand{\xdd}[1]{#1^{\vee \vee}}

\newcommand{\ho}[2]{Hom_{\C}(#1,#2)}

\newcommand{\ov}[1]{\overline{#1}}

\newcommand{\jag}[1]{Col_G({#1})/\m{G}_{#1}}

\newcommand{\im}{\textrm{im}}
\newcommand{\Bbbi}{\mathbbm{1}}

\renewcommand{\bf}{\bfseries}

\newcommand{\m}{\mathcal}

\newcommand\acl{\left\{ \begin{array}{ccc} }
\newcommand\acr{\end{array}\right\} }
\newcommand\crl{\left< \begin{array}{ccc} }
\newcommand\crr{\end{array}\right> }
\newcommand\brl{\left[ \begin{array}{ccc} }
\newcommand\brr{\end{array}\right] }
\newcommand\parl{\left( \begin{array}{ccc} }
\newcommand\parr{\end{array}\right) }

\newcommand{\fd}{\rightarrow}
\newcommand{\aut}{Aut_{\pt}(1_{\C})}
\newcommand{\sph}{Aut_{\pt}^{\pm}(1_{\C})}
\newcommand{\dc}{\Delta_{\C}}

\newcommand{\xr}[1]{\xrightarrow{#1}}

\newcommand{\cs}[1]{|#1|}

\newcommand{\grad}{\Gamma_{\mathcal{C}}}

\newcommand{\ap}{\mapsto}

\newcommand{\pf}{{\bf Proof : }}
\newcommand{\dis}[1]{\displaystyle{#1}}
\def\commutatif{\ar@{}[dll]|{\circlearrowleft}}
\def\commutatifg{\ar@{}[ddll]|{\circlearrowleft}}
\def\commutatifgg{\ar@{}[ddlll]|{\circlearrowleft}}
\def\commutatifs{\ar@{}[dddll]|{\circlearrowleft}}

\newcommand{\id}{\textrm{id}}
\newcommand{\torust}{S^1\times S^1\times S^1}

\numberwithin{equation}{section}

\title{Decomposition of the Turaev-Viro TQFT}
\author{J\'er\^ome Petit}
\address{Department of Mathematics, Tokyo institute of technology, 2-12-1 Ookayama Meguro-ku, 152-8551 Tokyo, Japan}
\email{petit.j.aa@titech.ac.jp}
\thanks{The author has been supported by JSPS Research Fellowship for Young Scientists.}

\begin{document}

\begin{abstract}
We show that for every spherical category $\C$ with invertible dimension, the Turaev-Viro TQFT admits a splitting into blocks which come from an HQFT, called the Turaev-Viro HQFT. The Turaev-Viro HQFT has the classifying space $B\grad$ as target space, where $\grad$ is a group obtained from the category $\C$. This construction gives a reformulation of the Turaev-Viro TQFT in terms of HQFT. Furthermore the Turaev-Viro HQFT is an extension of the \emph{homotopical Turaev-Viro invariant} which splits the Turaev-Viro invariant. An application of this result is a description of  the homological twisted version of the Turaev-Viro invariant in terms of HQFT.
\end{abstract}


\maketitle

\section{Introduction}

In the early 90's, a new \emph{quantum invariant} of 3-manifolds was introduced : the \emph{Turaev-Viro invariant} \cite{TV}. The original construction involves a quantum group at a root of unity. Barrett and Westburry \cite{BW} generalized the construction to \emph{spherical categories} with invertible dimension in a field $\Bbbk$. A spherical category is a semisimple sovereign category over a commutative ring $\Bbbk$ such that the left and right traces coincide. The dimension of a spherical category is the sum of squares of dimensions of simple
objects. The construction of the Turaev-Viro invariant consists in representing the 3-manifold by a triangulation, coloring the edges with simple objects of the spherical category and then assigning a 6j-symbol to each colored tetrahedron.

In \cite{Tu}, Turaev showed that the Turaev-Viro invariant extends to a \emph{TQFT}. In dimension 2+1, a TQFT assigns to every closed surface a finite dimensional vector space and to every three dimensional cobordism a linear map. More precisely, in dimension 2+1 a TQFT is a symmetric functor from the category of cobordisms of dimension 2+1 to the category of finite dimensional vector spaces. It can be interesting to extend the notion of TQFT to cobordisms and surfaces endowed with additional data. For instance, Blanchet, Habegger, Masbaum and Vogel \cite{BHMV} have showed that the Kauffman bracket extends to a TQFT for surfaces and cobordisms endowed with $p_1$-structures. In 2000, Turaev \cite{THQFT} defined a notion of TQFT for surfaces and cobordisms endowed with homotopy classes of continuous map to a \emph{target space} X,  called \emph{HQFT} (\emph{Homotopical Quantum Field Theory}). Turaev showed that a modular $G$-category, with $G$ an abelian group, gives rise to an HQFT with target space the Eilenberg-Maclane space $K(G,1)$. This HQFT is obtained from an invariant of the pair $(M,\xi)$, with $M$ a 3-manifold and $\xi\in H^1(M,G)$, which splits the Reshetikhin-Turaev invariant. Turaev described this HQFT in terms of other TQFT. The resulting HQFT splits as a product of a standard TQFT and a homological TQFT.

In the same spirit as the work of Turaev and Le \cite{TL} and \cite{THQFT}, we want to describe the Turaev-Viro TQFT in terms of other TQFTs and/or TQFTs with additional data. To fulfill this objective we will define an homotopical invariant called \emph{the homotopical Turaev-Viro invariant}. This invariant will extend to an HQFT called \emph{the Turaev-Viro HQFT}. In opposition to the work of Turaev \cite{THQFT},  we show that the Turaev-Viro TQFT comes from the Turaev-Viro HQFT. Roughly speaking, we obtain a decomposition of the Turaev-Viro invariant into blocks which come from the Turaev-Viro HQFT.

The Turaev-Viro invariant of a closed 3-manifold $M$ is a state-sum indexed by the colorings of a triangulation of $M$. The colorings of a triangulation $T$ are maps from the set of oriented 1-simplexes to the set of scalar objects (up to isomorphism) of a spherical category $\C$. The set of colorings of a triangulation $T$ is denoted $Col(T)$. The Turaev-Viro invariant is~:
$$
TV_{\C}(M)=\dc^{-n_0(T)}\sum_{c\in Col(T)}w_cW_c\in \Bbbk\, ,
$$
where $\dc$ is the dimension of the category, $n_0(T)$ is the number of 0-simplexes of $T$, $w_c$ is a scalar obtained from the coloring of the 1-simplexes and the trace of the category and $W_c$ is a scalar obtained from the 6j-symbols of the category. This invariant can be defined for 3-manifolds with boundary, in which case the Turaev-Viro invariant is a vector. Let $M$ be a 3-manifold with boundary $\Sigma$ and $T_0$ be a triangulation of $\Sigma$, the Turaev-Viro invariant is obtained from the following vector~:
\begin{equation*}
\scriptsize{TV_{\C}(M,c_0)=\dc^{-n_0(T)+n_0(T_0)/2}\sum_{c\in Col_{c_0}(T)}w_cW_c\in V_{\C}(\Sigma,T_0,c_0)\, ,}
\end{equation*}
where $Col(T)_{c_0}$ is the set of colorings of $T$ such that the restriction to $T_0$ is the coloring $c_0$ and $V_{\C}(\Sigma,c_0,T_0)$ is a vector space associated to the triple $(M,c_0,T_0)$. The Turaev-Viro invariant is $$\dis{TV_{\C}(M)=\sum_{c\in Col(T_0)}TV_{\C}(M,c_0)}\, .$$

To study the Turaev-Viro invariant and the TQFT obtained from it, we will assign to each spherical category $\C$  a group $\grad$, which comes from a universal graduation of the category. The group $\grad$ is called \emph{the graduator of $\C$}. To understand this group, the simplest case is the case of \emph{group categories}. If $\C$ is a group category then $\grad$ is the group of isomorphism classes of scalar objects. A group category is a semisimple tensor $\Bbbk$-category such that for every scalar object $X$ there exists an object $Y$ such that $X\pt Y\cong \Bbbi \cong Y\pt X$, with $\Bbbi$ the neutral element for the monoidal structure. For every spherical categories $\C$, we will use the  group $\grad$ to defined an homotopical invariant $HTV_{\C}$. This homotopical invariant will split the Turaev-Viro invariant. More precisely, we observe that every coloring $c$ of a triangulation $T$ of a closed 3-manifold $M$ leads to an homotopy class $x_c\in [M,B\grad]$, where $B\grad$ is the classifying space of the group $\grad$ and $[M,B\grad]$ is the set of homotopy classes of continuous map from $M$ to $B\grad$. These remarks lead to the following homotopical invariant of closed 3-manifolds~:

$$
HTV_{\C}(M,x)=\dc^{-n_0(T)}\sum_{\substack{c\in Col(T)\\x_c=x}}w_cW_c\, ,
$$
where $x\in [M,B\grad]$. The invariant $HTV_{\C}$ is \emph{the homotopical Turaev-Viro invariant}. We define this invariant for manifolds with boundary. Roughly speaking, for every 3-manifold $M$ with boundary $\Sigma$ endowed with a triangulation $T_0$ and for every coloring $c_0$ of $T_0$, we associate an homotopy class $x_{c_0}\in [\Sigma,B\grad]$ and from every coloring $c\in Col_{c_0}(T)$, we associate an homotopy class $x_c\in [M,B\grad]$ such that $x_{\Sigma}$ the homotopy class of the restriction of $x_c$ to $\Sigma$ is $x_{c_0}$. These remarks lead to the boundary version of the homotopical Turaev-Viro invariant~:
$$
HTV_{\C}(M,c_0,x)=\dc^{-n_0(T)+n_0(T_0)/2}\sum_{\substack{c\in Col_{c_0}(T)\\x_c=x\\x_{\Sigma}=x_{c_0}}}w_cW_c\, ,
$$
for every $x\in [M,B\grad]$ such that the homotopy of its restriction to $\Sigma$ is $x_{c_0}$.

The homotopical Turaev-Viro invariant splits the Turaev-Viro invariant~:
\begin{thm1}
Let $M$ be a 3-manifold, $\Sigma$ be the boundary of $M$ and $T_0$ be a triangulation of $\Sigma$. For every coloring $c_0\in Col(T_0)$, we have~:
$$TV_{\C}(M,c_0)=\sum_{\substack{x\in [M,B\grad]\\ x_{\Sigma}=x_{c_0}}}HTV_{\C}(M,c_0,x) \in V_{\C}(\Sigma,T_0,c_0)\,$$
and $HTV_{\C}(M,c_0,x)$ an invariant of the triple $(M,c_0,x)$. If the 3-manifold $M$ is closed, we obtain~:
$$
TV_{\C}(M)=\sum_{x\in [M,B\grad]}HTV_{\C}(M,x)\,.
$$
\end{thm1}

We prove that the homotopical Turaev-Viro invariant extends to an HQFT $\m{H}_{\C}$ with target space the classifying space $B\grad$. The HQFT $\m{H}_{\C}$ is called \emph{the Turaev-Viro HQFT}. The Turaev-Viro HQFT and the splitting given in Theorem \ref{thm:inv} leads to the main result of this article~:
\begin{thm2}
Let $\C$ be a spherical category. The Turaev-Viro TQFT  $\m{V}_{\C}$ is obtained from the Turaev-Viro HQFT $\m{H}_{\C}$~:
$$
\m{V}_{\C}(\Sigma)=\bigoplus_{x\in [\Sigma,B\grad]}\m{H}_{\C}(\Sigma,x)\,
$$
for every  closed surface $\Sigma$.
\end{thm2}

In the case of group categories defined for an abelian group this splitting is maximal in a sense that every block obtained from the above splitting is a one dimensional vector space :

\begin{promax}
  Let  $G$ an abelian group, $\alpha\in H^3(G,\Bbbk^*)$, $\C_{G,\alpha}$ be a group category, $g$ be a positive integer and $\Sigma_g$ be a closed surface of genus $g$, we have~:
  $$
  \m{V}_{\C_{G,\alpha}}(\Sigma_g)=\bigoplus_{x\in [\Sigma_g,BG]}\m{H}_{\C_{G,\alpha}}(\Sigma_g,x)\,
  $$
with $\m{H}_{\C_{G,\alpha}}(\Sigma_g,x)=\Bbbk$ for every $x\in[\Sigma_g,BG]$.
\end{promax}

An application of this work is a description of the homological twisted Turaev-Viro invariant, defined by Yetter \cite{homyetter} in terms of HQFT. The homological twisted Turaev-Viro invariant is an invariant for the pair $(M,\alpha)$, with $M$ a 3-manifold and $\alpha\in H_1(M,A)$, where $A$ is the group of monoidal automorphisms of the identity functor $1_{\C}$. The homological twisted Turaev-Viro invariant is given by the formula~:
$$
Y_{\C}(M,\alpha)=\dc^{n_0(T)}\sum_{c\in Col(T)}(\alpha:c)W_c\, ,
$$
where $(\alpha:c)$ is defined as follow~: we represent $\alpha$ by a map $b$ from the set of oriented edges of $T$ to $A$ and we set~:
$$
(\alpha: c)=\prod_{e\in T^1}b(e)_{c(e)}
$$
In the original paper of Yetter \cite{homyetter}, this invariant was defined for a semisimple braided $\Bbbk$-tensor category. We extend the construction to spherical categories and we prove that the invariant $Y_{\C}$ comes from an HQFT.
\begin{thm3}
Let $\C$ be a spherical category, $M$ be a 3-manifold. For every $h\in H_1(M,\aut)$, we have~:
$$
Y_{\C}(M,h)=\sum_{x\in[M:B\grad]}(h:x)HTV_{\C}(M,x)\, ,
$$
with $\dis{(h:x)=\prod_{e\in T^1}\alpha^e(c(e))}$, $\alpha$ a representative of $h$ and $c\in Col_x(T)$.
\end{thm3}
The rest of the paper is organized as follows. In Section \ref{sec:rappel}, we review several facts about
monoidal categories and we define the universal graduation of semisimple tensor categories. Section \ref{sec:TVconstruc} recalls the construction of the Turaev-Viro invariant. In Section \ref{sec:HTVconstruc}, we define the homotopical Turaev-Viro invariant and we prove that the homotopical Turaev-Viro invariant splits the Turaev-Viro invariant (Theorem \ref{thm:inv}). In Section \ref{sec:calcul}, we compute the homotopical Turaev-Viro invariant for the sphere $S^3$, the 3-torus $S^1\times S^1\times S^1$ and lens spaces. We compute for group categories and the quantum group $U_q(\frak{sl}_2)$ with $q$ a root of unity. In Section \ref{sec:splitting}, we show that for every spherical category $\C$ the Turaev-Viro TQFT comes from an HQFT with target space the classifying space $B\grad$ (Theorem \ref{HQFT}). The proof is in two steps. First, we show that the homotopical Turaev-Viro invariant splits the Turaev-Viro TQFT into blocks. Then we show that these blocks come from an HQFT. We show that the splitting obtained is maximal in the case of group categories defined for an abelian group (Proposition \ref{pro:max}). We end this section with a computation of the HQFT for the torus $S^1\times S^1$ in the case of the quantum group $U_q(\frak{sl}_2)$ with $q$ a root of unity. In section \ref{sec:twist}, we reformulate the homological twisted Turaev-Viro invariant in terms of HQFT.  In Section \ref{sec:table}, we give some values of the homotopical Turaev-Viro invariant for group categories.

\subsection*{Notations and conventions}
Throughout this paper, $\Bbbk$ will be a commutative, algebraically closed and characteristic zero field. Unless otherwise specified, categories are assumed to be small and monoidal categories are assumed to be strict. We denote by $\Bbbi_{\C}$ the unit object of monoidal category $\C$. If there is no ambiguity on the choice of the category then we denote by $\Bbbi$ the unit object.

If $\C$ is a monoidal category, we denote by $\aut$ the abelian group of monoidal automorphisms of the identity functor $1_{\C}$.

Throughout this paper, we use the following notation. For an oriented manifold $M$, we denote by $\overline{M}$ the same manifold with the opposite orientation.

\section{Categories}\label{sec:rappel}

In this section, we review a few general facts about categories with structure, which we use intensively throughout this text.

\subsection*{Autonomous categories}

Let $\C$ be a monoidal category. A \emph{duality} of $\C$ is a data $(X,Y,e,h)$, where $X$ and $Y$ are objects of $\C$ and $e : X\pt Y \fd \Bbbi$ (\emph{evaluation}) and  $h : \Bbbi \fd Y\pt X$ (\emph{coevaluation}) are morphisms of $\C$, satisfying~:
$$
(e\pt \id_X)(\id_X\pt h)=\id_X\qquad \mbox{and}\qquad (\id_Y\pt e)(h\pt \id_Y)=\id_Y\,.
$$
If $(X,Y,e,h)$ is a duality, we say that $(Y,e,h)$ is \emph{a right dual of $X$}, and $(X,e,h)$ is \emph{a left dual of $Y$}. If a right or left dual of an object exists, it is unique up to unique isomorphism.

A \emph{right autonomous} (resp. \emph{left autonomous}, resp. \emph{autonomous}) category is a monoidal category for which every object admits a right dual (resp. a  left dual, resp. both a left and a right dual). In the literature, autonomous categories are also called rigid categories.

If $\C$ has right duals, we may pick a right dual $(\xd{X},e_X,h_X)$ for each object $X$. This defines a monoidal functor $\xd{?} : \C^{op}\fd \C$, where $\C^{op}$ denotes the category with opposite composition and tensor products. This monoidal functor is called \emph{right dual functor}. Notice that the actual choice is innocuous, in the sense that different choices of right duals define canonically isomorphic right dual functors.

Similarly a choice of left duals $(\xg{X},\epsilon_X,\eta_X)$ for each object $X$ defines a monoidal functor $\xg{?} : \C^{op}\fd\C$, called the \emph{left dual functor}.

In particular, the right dual functor leads to the \emph{double right dual functor} $\xdd{?} : \C\fd \C$ defined by $X\ap \xdd{X}$ and $f\ap \xdd{f}$, which is a monoidal functor.

\medskip

\subsection*{Sovereign categories}

A \emph{sovereign structure} on a right autonomous category $\C$ consists in the choice of a right dual for each object of $\C$ together with a monoidal isomorphism $\phi : 1_{\C} \fd \xdd{?}$, where $1_{\C}$ is the identity functor of $\C$. Two sovereign structures are \emph{equivalent} if the corresponding monoidal isomorphisms coincide via the canonical identification of the double dual functors.

A \emph{sovereign category} is a right autonomous category endowed with an equivalence class of sovereign structures.

Let $\C$ be a sovereign category, with chosen right duals $(\xd{X},e_X,h_X)$ and sovereign isomorphisms $\phi_X : X\fd \xdd{X}$. For each object $X$ of $\C$, we set :

$$
\epsilon_X = e_{\xd{X}}(\id_{\xd{X}}\pt \phi_X) \qquad \textrm{and} \qquad \eta_X = (\phi_X^{-1}\pt\id_{\xd{X}})h_{\xd{X}}\, .
$$

Then $(\xd{X},\epsilon_X,\eta_X)$ is a left dual of $X$. Therefore $\C$ is autonomous. Moreover the right left functor $\xg{?}$ defined by this choice of left duals coincides with $\xd{?}$ as a monoidal functor. From now on, for each sovereign category $\C$ we will make this choice of duals.

\medskip

The sovereign structures on a sovereign category are given by the group $\aut$.

\begin{pro}\label{pro:souv}
 Let $\C$ be a sovereign category and $\phi_0$ be the sovereign structure. The map $$\phi \ap \phi_0^{-1}\phi$$ is a bijection between the set of sovereign structures on $\C$ and the group $\aut$ of monoidal automorphisms of the functor identity $1_{\C}$.
\end{pro}

The sovereign categories are an appropriate categorical setting for a good notion of trace. Let $\C$ be a sovereign category and $X$ be an object of $\C$. For each endomorphism $f\in Hom_{\
C}(X,X)$,
$$\tr_l(f)=\epsilon_X (\id_{\xd{X}}\pt f)h_X \in Hom_{\C}(\Bbbi,\Bbbi)$$
is the \emph{left trace} of $f$ and
$$\tr_r(f)=e_X(f\pt \id_{\xd{X}})\eta_X\in Hom_{\C}(\Bbbi,\Bbbi)$$
is the \emph{right trace} of $f$. We denote by $\dim_r(X)=\tr_r(\id_X)$ (resp. $\dim_l(X)=\tr_l(\id_X)$) the \emph{right dimension} (resp.
\emph{left dimension}) of $X$.

\subsection*{Tensor categories}

By a \emph{$\Bbbk$-linear category}, we shall mean a category for which the set of morphisms are $\Bbbk$-spaces, the composition is $\Bbbk$-bilinear
, there exists a null object and for every objects $X$, $Y$ the direct sum $X\oplus Y$ exists in $\C$.

A $\Bbbk$-linear category is \emph{abelian} if it admits finite direct sums, every morphism has a kernel and a cokernel, every monomorphism is the kernel of its cokernel, every epimorphism is the cokernel of its kernel, and every morphism is expressible as the composite of an epimorphism followed by a monomorphism.

An object $X$ of an abelian $\Bbbk$-category $\C$ is \emph{scalar} if $Hom_{\C}(X,X)\cong \Bbbk$.

A \emph{tensor category over $\Bbbk$} is an autonomous category endowed with a structure of $\Bbbk$-linear abelian category such that the tensor product is $\Bbbk$-bilinear and the unit object is a scalar object.

A $\Bbbk$-linear category is \emph{semisimple} if :
\begin{itemize}
\item[(i)] every object of $\C$ is a finite direct sum of scalar objects,
\item[(ii)] for every scalar objects $X$ and $Y$, we have : $X\cong Y$ or $\ho{X}{Y}=0$\,.
\end{itemize}

A \emph{semisimple $\Bbbk$-category} is a semisimple abelian $\Bbbk$-linear category and every simple object is a scalar object. Notice that in a semisimple abelian $\Bbbk$-category every scalar object is a simple object. By a \emph{finitely semisimple $\Bbbk$-category} we shall mean a semisimple $\Bbbk$-category which has finitely many isomorphism classes of scalar objects. The set of isomorphism classes of scalar objects of an abelian $\Bbbk$-category $\C$ is denoted by $\lc$.

%

\subsection*{Graduations}

Let $\C$ be semisimple tensor $\Bbbk$-category and $G$ be a group. A \emph{$G$-graduation of $\C$} is a map $p : G\fd \lc$ satisfying :
\begin{itemize}
\item[$\bullet$] $p(Z)=p(X)p(Y)$, for every scalar objects $X,Y,Z$ such that $Z$ is a subobject of $X\pt Y$.
\end{itemize}

A \emph{graduation of $\C$} is a pair $(G,p)$, where $G$ is group and $p$ is a $G$-graduation of $\C$.

By induction, the multiplicity property of a graduation can be extended to $n$-terms.

\begin{pro}\label{pro:graduation}
Let $\C$ be a semisimple tensor $\Bbbk$-category. There exists a graduation $(\grad,\cs{?})$ of $\C$ satisfying the following universal property : for every graduation $(G,p)$ of $\C$, there exists a unique group morphism $f : \grad \fd G$ such that the diagram~:
$$
\xymatrix{\lc \ar[rr]^{\cs{?}} \ar[dr]_{p} && \grad \ar[dl]^{f}\commutatif\\
&G& }
$$
commutes.
\end{pro}
\pf

To build $\grad$, we define an equivalence relation $\sim$ on $\lc$. Let $X$ and $Y$ be two scalar objects; $X\sim Y$ if and only if there exists a finite sequence of scalar objects $T_1,..., T_n$ such that $X$ and $Y$ are subobjects of $T_1\pt ... \pt T_n$. This relation is reflexive and symmetric. Let us show that this relation is transitive. Let $X,Y,Z$ be scalar objects such that $X\sim Y$ and $Y\sim Z$. Thus there exists two tensor products of scalar objects $A$ and $B$ such that $X$ and $Y$ are subobjects of $A$, and  $Y$ and $Z$ are subobjects of $B$. To prove the equivalence $X\sim Z$, we will show that $A$ and $B$ are subobjects of $B\pt \xd{B} \pt A$. Since $\Bbbi$ is a subobject of $B \pt \xd{B}$, it follows that $A$ is a subobject of $B\pt \xd{B}\pt A$. For every scalar object $X$, there exists a right dual $\xd{X}$ and a left dual $\xg{X}$. Since the category $\C$ is semisimple tensor $\Bbbk$-category, $\xd{X}$ and $\xg{X}$ are isomorphic. It follows that : $B\cong B\pt \Bbbi\hookrightarrow B\pt \xd{X}\pt X \cong B\pt \xg{X}\pt X$ and thus $B$ is a subobject of $B\pt \xd{B}\pt A$.

We denote by $\grad$ the quotient of $\lc$ by $\sim$ and by $\cs{?} : \lc \fd \grad$ the canonical surjection. We define an internal law on $\grad$ in the following way : if $X,Y$ are scalar objects of $\C$ and $Z$ is a scalar subobject of $X\pt Y$, we set $\cs{X}.\cs{Y}=\cs{Z}$. This law is well defined, associative and $\cs{\Bbbi}$ is the neutral element. For every scalar object $X$, $\cs{\xd{X}}$ is the inverse of $\cs{X}$, thus $\grad$ is a group and by construction $\cs{?}$ is a $\grad$-graduation of $\C$.

Let us show the universal property of $(\grad,\cs{?})$. Let $(G,p)$ be a graduation of $\C$, if the scalar objects $X$ and $Y$ equivalent for the relation $\sim$ then there exists a tensor product of scalar objects $A=T_1\pt ... \pt T_n$ such that $X$ and $Y$ are subobjects of $A$. Since $p$ is a $G$-graduation, we have : $p(X)=p(Y)=p(T_1)....p(T_n)$. Thus there exists a unique map $f : \grad \fd G$ such that $f\cs{?}=p$. For every scalar objects $X,Y$ and for every scalar subobject $Z$ of $X\pt Y$, we have $f(\cs{X})f(\cs{Y})=p(X)p(Y)=p(Z)=f(\cs{Z})=f(\cs{X}\cs{Y}),$ thus $f$ is a group morphism.

\qed

Let $\C$ be a semisimple tensor $\Bbbk$-category, the group $\grad$ which defines the universal graduation $(\grad,\cs{?})$ is called \emph{the graduator of $\C$}.

\begin{pro}\label{pro:gradaut}
Let $\C$ be a semisimple tensor $\Bbbk$-category. There exists a canonical isomorphism between $\aut$ and the group of group morphisms from $\grad$ to $\Bbbk^*$.
\end{pro}
\pf
Let $\phi\in \aut$, for every scalar object $X$, we have $\phi_X =\epsilon_X \id_X$, with $\epsilon_X \in \Bbbk^*$. Since $\phi$ is a monoidal morphism, we have for every scalar objects $X, Y$, $\phi_{X\pt Y}=\phi_X\pt \phi_Y$. It follows that $\epsilon$ is a $\Bbbk^*$-graduation of $\C$. Conversely, let $p$ be a $\Bbbk^*$-graduation of $\C$. Since $\C$ is a semisimple tensor $\Bbbk$-category, $\phi_X=p(X)\id_X$, for every scalar object $X$, defines a natural isomorphism $\phi : 1_{\C} \fd 1_{\C}$. Moreover, we have $\phi_{X\pt Y}=\phi_X\pt \phi_Y$ for every scalar objects $X,Y$, thus $\phi$ is a monoidal isomorphism. Using the universal property of the group $\grad$ we can conclude.

\qed

It follows that the graduator $\grad$ of a sovereign category $\C$ describes the sovereign structures of this category.

\subsection*{Examples}
\subsubsection*{Group categories}
Let $\C$ be a monoidal category. An object $X$ of $\C$ is \emph{invertible} if there exists an object $Y$ of $\C$ (\emph{inverse of $X$}) such that $X\pt Y \cong \Bbbi \cong Y\pt X$. Notice that in a right autonomous category if $Y$ is the inverse of $X$ then $Y$ is isomorphic to $\xd{X}$. The group of invertible elements of $\C$ is called \emph{the Picard group of $\C$}.

A \emph{group category} is a semisimple tensor $\Bbbk$-category such that every scalar object is invertible. The following theorem gives a classification of group categories~:

\begin{theorem}[section 7.5 \cite{FK}]\label{thm:FK}
The datum of a group category is equivalent to the data of a finite group $G$ and a cohomology class $\alpha\in H^3(G,\Bbbk^*)$.
\end{theorem}

More precisely, up to monoidal equivalence a group category $\C$ is the category of $\lc$-graded finite dimensional vector spaces. Morphisms are linear maps which preserve the grading and the associativity constraint is given by an element of the cohomology group $H^3(\lc,\Bbbk^*)$. From now on, the group category defined by the group $G$ and the cohomological class $\alpha\in H^3(G,\Bbbk^*)$ will be denoted $\C_{G,\alpha}$.

Let $\C_{G,\alpha}$ be a group category, the tensor product of two scalar objects is a scalar object, thus two scalar objects of $\C$ are equivalent for the equivalence relation $\sim$ if and only if they are isomorphic. It follows~: $\Gamma_{\C_{G,\alpha}}=\Lambda_{\C_{G,\alpha}}=G$.

\subsubsection*{$U_q(\mathfrak{sl}_2)$, with $q$ generic}
Let $q$ be a complex parameter which is not a root of unity. The scalar objects of the category of the representations of $U_q(\mathfrak{sl}_2)$ are given by the positive integers $\{0, 1,2 , ..., i , ...\}$, where $i$ denotes the unique (up to isomorphism) scalar of dimension $i+1$ (e.g.  \cite{6j}, \cite{Kassel}).  The tensor product of scalar objects is given by the Clebsch-Gordan formula (e.g. \cite{6j}, \cite{Kassel} )~:
\begin{equation}\label{CB}
i\pt j=i+j\oplus i+j-2\oplus... \oplus \cs{i-j}\, .
\end{equation}
Thus for every $k\in \N$, the objects $2k$ and $0$ are subobjects of the tensor product $k\pt k$ and the objects $2k+1$ and $1$ are subobjects of the tensor product $2k+1\pt 2k$. Thus for every $k\in \N$, we have : $2k\sim 0$ and $2k+1\sim 1$, and it follows from the Clebsch-Gordan formula (\ref{CB}) that $0$ and $1$ cannot be subobjects of the same tensor product of scalar objects. It results~: $\Gamma_{U_q(\mathfrak{sl}_2)}=\Zz_2$.

\subsubsection*{$U_q(\mathfrak{sl}_2)$, with $q$ root of unity}
Let $A$ be a $2r$-th primitive root unity such that $A^2=q$ is a $r$-th primitive root of unity. The scalar objects of the category of the finite dimensional representations of $U_q(\mathfrak{sl}_2)$ are given by the integers $\{0,1,...,r-2\}$ (e.g. \cite{6j}, \cite{Kassel},  \cite{Roberts}), where  $i$ denotes the unique (up to isomorphism) scalar representation of dimension $i+1$. In this category, the tensor product of scalar objects is given by the Clebsch-Gordan formula :
$$
i\pt j=\bigoplus_{\substack{k\\(i,j,k)\mbox{ is admissible }}}k\, ,
$$
where an admissible triple $(i,j,k)$ is the data of three positive integers $i,j$ and $k$ satisfying~:
\begin{itemize}
\item $i\leq j+k$, $j\leq i+k$ and $k\leq j+i$\, ,
\item $i+j+k\leq 2(r-2)$\, ,
\item $i+j+k$ is even\, .
\end{itemize}
Similarly to the generic case, we show  that there are only two elements in $\Gamma_{U_q(\mathfrak{sl}_2)}$~:~ the equivalence class of $0$ and the equivalence class of $1$. Thus the graduator $\Gamma_{U_q(\mathfrak{sl}_2)}$ is $\Zz_2$.

\subsection*{Spherical categories}A \emph{spherical category} is a sovereign, finitely semisimple tensor $\Bbbk$-category satisfying~:
\begin{itemize}
\item[$\bullet$] for every object $X$ of $\C$ and for every morphism $f : X \fd X$ : $\tr_r(f)=\tr_l(f)$.
\end{itemize}

A \emph{spherical structure} on $\C$ is a sovereign structure on $\C$ such that $\C$ is a spherical category.

From now on, for every spherical category the left and right trace (resp. dimension) will be denoted by $\tr$ (resp. $\dim$).

The dimension of a spherical category is the scalar : $\dis{\dc=\sum_{X\in \lc}\dim(X)^2}\in \Bbbk$.

\begin{pro}\label{pro:sphgrad}
Let $\C$ be a spherical category ; we denote by $\phi_0$ the spherical structure. The map
\begin{equation*}
\phi \ap \phi_0^{-1}\phi
\end{equation*}
is a bijection between the set of spherical structures of $\C$ and the group $\sph=\{\psi \in \aut \mid \mbox{for every scalar object}$ $\mbox{ $X$, } \psi_X=\pm \id_X\}$\index{$\sph$}.
\end{pro}
\begin{proof} Let  $\phi_0$ and $\phi$ be spherical structures, we denote by $\tr_r,\tr_l,\dim_r,\dim_l$ (resp. $\tr_l^0,\tr_r^0,\dim_r^0,\dim_l^0$) the right and left traces  and the left and right dimensions  defined by the spherical structure $\phi$ (resp. $\phi_0$). Let $X$ be a scalar object of $\C$, then : $(\phi_0^{-1}\phi)_X=\lambda_X \id_X$ with $\lambda_X\in \Bbbk^{\star}$, it follows :
\begin{align*}
\dim_l(X)&= e_{\xd{X}}(\id_{\xd{X}}\pt \phi_X )h_X\\
&=\tr^0_l(\phi_{0_X}^{-1}\phi_X) \\
&=\lambda_X \dim^0_l(X),
\end{align*}
and
\begin{align*}
\dim_r(X)&= e_{X}(\phi_X^{-1}\pt \id_{\xd{X}})h_{\xd{X}}\\
&=\tr_r^0(\phi_X^{-1}\phi_{0_X}) \\
&= \lambda_X^{-1} \dim_r^0(X),
\end{align*}
then we have : $\lambda_X^{-1}\dim_r^0(X)=\lambda_X \dim^0_r(X)$. The object $X$ is a scalar object thus the right dimension is invertible, it follows : $\lambda_X^2=1$ .  Thus we have built a map from the set of spherical structures to the group $\sph$ :
\begin{align*}
\{\mbox{spherical structures on $\C$} \} & \fd \sph \\
\phi & \ap \phi_0^{-1}\phi \, .
\end{align*}
The map is injective. Let us show that this map is bijective. Let $\Psi \in \sph\subset \aut$. According to the proposition \ref{pro:souv}, we know that there exists a unique sovereign structure  $\phi$ such that $\Psi=\phi_0^{-1}\phi$.  The category $\C$ is a semisimple $\Bbbk$-category, to show that $\phi$ is a spherical structure we must show that the left and right trace coincide for every endomorphism of scalar objects. It is equivalent to show that for every scalar object $X$, the right and left dimension of $X$ are equals. Set $\dim_l$ and $\dim_r$ (resp. $\dim_l^0$ and $\dim_r^0$) the left and right dimension defined by the sovereign structure $\phi$ (resp. the spherical structure $\phi_0$) and $\tr_r^0$, $\tr_l^0$ are the left and right trace defined by  $\phi_0$ :
\begin{align*}
\dim_l(X)&=e_{\xd{X}}(\id_{\xd{X}}\pt \phi_X )h_X \\
&= e_{\xd{X}}(\id_{\xd{X}}\pt {\phi_0}_X{\phi_0}_X^{-1}\phi_X )h_X\\
&=\tr^0_l({\phi_0}_X^{-1}\phi_X)\\
&=\tr_r^0({\phi_0}_X^{-1}\phi_X)\\
&= e_{X}({\phi_0}_X^{-1}\phi_X {\phi_0}_X^{-1}\pt \id_{\xd{X}})h_{\xd{X}}\\
&=e_{X}(\phi^{-1}_X{\phi_0}_X {\phi_0}_X^{-1}\pt \id_{\xd{X}})h_{\xd{X}}\\
&=e_{X}(\phi^{-1}_X\pt \id_{\xd{X}})h_{\xd{X}}\\
&=\dim_r(X).
\end{align*}
The sixth equality come from the fact that for every scalar object $X$, we have the following relation~: ${\phi_0}_X^{-1}\phi_X=\pm \id_X=\phi_X^{-1}{\phi_0}_X$.

\end{proof}

\begin{pro}\label{pro:gradsph}
Let $\C$ be a spherical category. There is a bijection between the set $\sph=\{\Psi \in \aut\mid \textrm{for every scalar object} X : \Psi_X=\pm\id_X\}$ and the set of group morphisms from $\grad$ to the group $\{\pm 1\}$.
\end{pro}
\begin{proof} The proof is the same as the proof of the proposition \ref{pro:gradaut} \end{proof}

\section{Turaev-Viro invariant}\label{sec:TVconstruc}

In this section, we recall the construction of the Turaev-Viro invariant. Throughout this section $\C$ will be a spherical category

\medskip

An \emph{orientation} of a $n$-simplex $F$ is a map  $o : Num(F)\fd \{\pm1\}$, where $Num(F)$ is the set of numberings of $F$, invariant under the action of the alternated group $\mathfrak{A}_{N+1}\subset \mathfrak{S}_{N+1}$.

Let $T$ be an oriented simplicial complex, we denote the set of oriented $p$-simplexes by $T^p_o$. A \emph{coloring} of $T$ is a map $c : T^1_o \fd \lc$ satisfying : \begin{itemize}
\item[(i)] $c(x_1x_2)=c(x_2x_1)^{\vee}$, for every oriented 1-simplex $(x_1x_2)$,
\item[(ii)] the unit object $\Bbbi$ is a subobject of $c(x_1x_2)\pt c(x_2x_3)\pt c(x_3x_1)$ for every oriented 2-simplex $(x_1x_2x_3)$.
\end{itemize}
We denote by $Col(T)$ the set of colorings of $T$ .

Let $f$ be an oriented 2-simplex, $c$ be a coloring of $T$ and $\nu=(x_1x_2x_3)$ be a numbering of $f$ compatible with the orientation of $f$. Set :
$$
V_{\C}(f,c)_{\nu}=\ho{\Bbbi}{c(x_1x_2)\pt c(x_2x_3)\pt c(x_3x_1)}\, .
$$
The vector space $V_{\C}(f,c)$ does not depend on the choice of the numbering compatible with the orientation (e.g. \cite{BW}, \cite{GK}, \cite{Tu}). From now on the vector space $V_{\C}(f,c)_{\nu}$, with $\nu=(x_1x_2x_3)$, will be denoted by $V_{\C}(x_1x_2x_3,c)$. If there is no ambiguity on the choice of the coloring $c$, then $V_{\C}(x_1x_2x_3,c)$ will be denoted by $V_{\C}(x_1x_2x_3)$.

Let us recall some properties of the vector space defined above. For every scalar objects $X$, $Y$ and $Z$, we set~:
\begin{align}
\omega_{\C} : Hom_{\C}(\Bbbi,X\pt Y \pt Z)\pt_{\Bbbk} Hom_{\C}(\Bbbi,\xd{Z}\pt \xd{Y} \pt \xd{Z})&\fd \Bbbk^* \label{bilinear}\\
f\pt g & \ap tr(\xd{f}g)\, .\nonumber
\end{align}
For every spherical category $\C$, the bilinear form $\omega_{\C}$ is non degenerate (e.g. \cite{BW}, \cite{GK}, \cite{Tu}). Let $f$ be an oriented 2-simplex, we denote by $\ov{f}$ the 2-simplex $f$ endowed with the opposite orientation. Let $c$ be a coloring of $f$, the bilinear form (\ref{bilinear}) induces~: $V_{\C}(f,c)^{*}\cong V_{\C}(\ov{f},c)$.

In the construction of the Turaev-Viro invariant, we assign to every oriented 3-simplex a vector which lies in the vector space defined by the faces of the 3-simplex. The vector assigned to each 3-simplex is obtained by the 6j-symbols of the category. More precisely, let $T$ be a triangulation of a 3-manifold $M$ and $c$ be a coloring of $T$. We associate to every oriented 3-simplex $(x_1x_2x_3x_4)$ a vector $L^{\pm}((x_1x_2x_3x_4),c)$ defined by the 6j-symbols of the category $\C$~:
\begin{align*}
&L^{+}((x_1x_2x_3x_4),c)\in V_{\C}((x_2x_3x_4),c)\pt_{\Bbbk}V_{\C}((x_1x_4x_3),c)\pt_{\Bbbk}V_{\C}((x_1x_2x_4),c)\pt_{\Bbbk}V_{\C}((x_1x_3x_2),c)\, ,\\
&\mbox{if $(x_1x_2x_3x_4)$ and $M$ have the same orientation.}\\
&L^{-}((x_1x_2x_3x_4),c)\in V_{\C}((x_2x_4x_3),c)\pt_{\Bbbk}V_{\C}((x_1x_3x_4),c)\pt_{\Bbbk}V_{\C}((x_1x_4x_2),c)\pt_{\Bbbk}V_{\C}((x_1x_2x_3),c)\,.\\
&\mbox{if $(x_1x_2x_3x_4)$ and $M$ have opposite orientations.}
\end{align*}

Since the category $\C$ is spherical, the vector $L^{+}((x_1x_2x_3x_4),c)$ (resp. $L^{+}((x_1x_2x_3x_4),c)$) does not depend on the choice of the numbering which respects the orientation of $(x_1x_2x_3x_4)$.

Let us explain the final step of the construction of the Turaev-Viro invariant. Let $T$ be a triangulation of 3-manifold $M$. Every 2-simplex contained in the interior of $M$ is the intersection of a unique pair of 3-simplexes (Fig. \ref{des:intersection}) (e.g. \cite{BW}, \cite{GK}, \cite{Tu}).

\begin{figure}[h!]
$$
\vcenter{\hbox{\includegraphics[height=2cm,width=2.5cm]{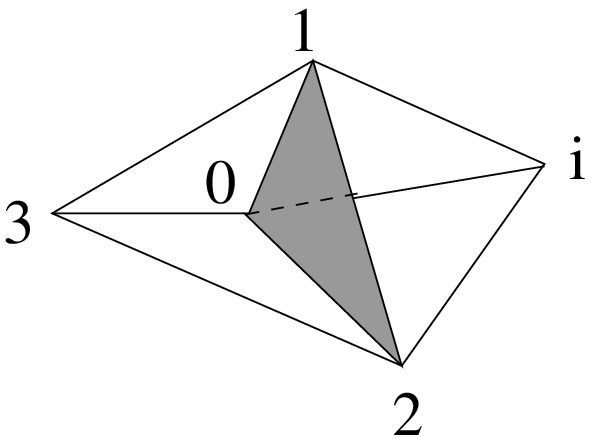}}}
$$
\caption{}
\label{des:intersection}
\end{figure}

We set a numbering of this pair of 3-simplexes. We can assume that the 3-simplex $(0123)$ has the same orientation of the manifold $M$. Thus the 3-simplex $(i012)$ (resp. $(0i12)$) has the same (resp. opposite) orientation of $M$. We obtain the following vectors~:
\begin{align*}
 L^{+}((0123),c)\in V_{\C}((123),c)\pt_{\Bbbk}V_{\C}((032),c)\pt_{\Bbbk}V_{\C}((013),c)\pt_{\Bbbk}V_{\C}((021),c)\,.\\
 L^{+}((i012),c)\in V_{\C}((012),c)\pt_{\Bbbk}V_{\C}((i21),c)\pt_{\Bbbk}V_{\C}((i02),c)\pt_{\Bbbk}V_{\C}((i10),c)\,.
\end{align*}
We remark that the vector $L^{+}((0123),c)$ has a component in the vector space $V_{\C}((021),c)$ and the vector $L^{+}((i012),c)$ has a component in the vector space $V_{\C}((012),c)$. Recall that the vector space $V_{\C}((021),c)$ is the dual vector space of $V_{\C}((012),c)$ for the pairing (\ref{bilinear}). For this pair of vectors, we can apply the pairing (\ref{bilinear}) on the tensor product $L^{+}((0123),c)\pt L^{+}((i012),c)$~:
\begin{align*}
 V_{\C}((012),c)\pt V_{\C}((021),c)&\fd \Bbbk\\
f\pt g&\ap \tr(\xd{f}g)\, .
\end{align*}

If we consider the 3-simplex $(i012)$ with the opposite orientation, then we obtain the following vector~:
\begin{align*}
 L^{-}((0i12),c)\in V_{\C}((i21),c)\pt_{\Bbbk}V_{\C}((012),c)\pt_{\Bbbk}V_{\C}((021),c)\pt_{\Bbbk}V_{\C}((0i1),c)\,.
\end{align*}
We can do the same operation. Thus if two 3-simplexes have a common face, we can do this operation for the vector space assigned to the common face of the pair of 3-simplexes. Since every 2-simplex in the interior of a 3-manifold is the common face of a pair of 3-simplexes, for every coloring $c$ of $T$ this operation for every 3-simplexes of $T$ leads to a scalar if the manifold $M$ is without boundary or to a vector in $\dis{\bigotimes_{f\in T^2_{\partial M} }V_{\C}(f,c)}$ if the manifold $M$ has a boundary $\partial M$. We denote this vector (or scalar)  by $W_c$.

  We introduce some notations. Let $\Sigma$ be an oriented closed surface endowed with a triangulation $T_0$. For every coloring $c_0$  of $T_0$, we set : $\dis{V_{\C}(\Sigma,T_0,c_0)=\bigotimes_{f\in T^2}V_{\C}(f,c_0)}$ and $\dis{V_{\C}(\Sigma,T_0)=\bigoplus_{c\in Col(T_0)}V_{\C}(\Sigma,T_0,c)}$. Let $M$ be 3-manifold with boundary $\Sigma$ and $T$ be a triangulation of $M$ such that its restriction to $\Sigma$ is $T_0$. For every coloring $c_0\in Col(T_0)$, we denote by $Col_{c_0}(T)$ the set of colorings of $T$ such that the restriction to $T_0$ is $c_0$. With this notation, for every coloring $c\in Col_{c_0}(T)$, we have~: $W_c\in V_{\C}(\Sigma,T_0,c_0)$. Furthermore we choose a square root $\dc^{1/2}$ of $\dc$.

For every scalar object $X$ of $\C$, we set $\dim(X)^{1/2}$ a square root of $\dim(X)$. The equalities $\dim(X)^{1/2}= \dim(\xd{X})^{1/2}$ and $\dim(X)= \dim(\xd{X})$ ensure independence of $\dim(c(e))$, $\dim(c(e))^{1/2}$ of the choice of the orientation of $e$, for every coloring $c$.

\begin{theorem}[Turaev-Viro invariant \cite{BW}, \cite{GK}, \cite{Tu}, \cite{TV} ]
Let $\C$ be a spherical category with an invertible dimension, $M$ be a compact oriented 3-manifold and $\partial M$ be the boundary of $M$ endowed with a triangulation $T_0$. For every coloring $c_0\in Col(T_0)$, we set :
\begin{equation}
\label{invTV}
\scriptsize{TV_{\C}(M,c_0)=\dc^{-n_0(T)+n_0(T_0)/2}\sum_{c \in Col_{c_0}(T)}\prod_{e\in T^1_0}\dim(c_0(e))^{1/2}\prod_{e\in
T^1\backslash T^1_0}\dim(c(e))W_{c}\in V(\partial M,c_0T_0)\, ,}
\end{equation}
where $n_0(T)$ (resp. $n_0(T_0)$) is the number of $0$-simplexes of $T$ (resp. $T_0$) and $T^1\backslash T^1_0$ is the set of 1-simplexes of $M\backslash\partial M$.
For every coloring $c_0\in Col(T_0)$, the vector $TV_{\C}(M,c_0)$ is independent on the choice of the triangulation of $M$ which extends $T_0$. The Turaev-Viro invariant is the vector~:
$$
TV_{\C}(M)=\sum_{c_0\in Col(T_0)}TV_{\C}(M,c_0)\in V_{\C}(\partial M,T_0)\,.
$$
\end{theorem}

From now on, for every coloring $c\in Col_{c_0}(T)$ we denote by $w_c$ the scalar $\dis{\prod_{e\in T^1_0}\dim(c_0(e))^{1/2}\prod_{e\in T^1\backslash T^1_0}\dim(c(e))}$.

There exists other normalization of the Turaev-Viro invariant. For every cobordism $M$ such that $\partial M=\overline{\Sigma_+}\coprod \Sigma_-$, with $T_+$ (resp. $T_-$) a triangulation of $\Sigma_+$ (resp. $\Sigma_-$), we can replace the scalar $\dc^{(n_0(T_+)+n_0(T_-))/2}$ by $\dc^{n_0(T_+)}$ or $\dc^{n_0(T_-)}$ .  Turaev and Viro (\cite{Tu}, \cite{TV}) use the scalar $\dc^{(n_0(T_+)+n_0(T_-))/2}$ to normalize the Turaev-Viro invariant. Notice that with this normalization, we don't have to choose a square root of $\dc$. The Turaev-Viro TQFTs obtained from these changes of normalization are the same (up to isomorphism).

\section{Homotopical Turaev-Viro invariant}\label{sec:HTVconstruc}

\subsection{Fundamental groupo\"{i}d}

We recall the definition of the fundamental groupo\"{i}d and we set some notations.

Let $T$ be a simplicial complex. A \emph{path} of $T$ is a finite sequence $v_0v_1...v_n$ of 0-simplexes of $T$ in which each consecutive pair $v_iv_{i+1}$ spans a 1-simplex of $T$. In particular a 0-simplex is a path. Let $v$ be a 0-simplex of $T$, a \emph{loop based at $v$} is a path $v_0v_1...v_n$ such that $v_0=v_n=v$. We consider two paths to be \emph{equivalent} if we can obtain one from the other by a finite number of operations of the following type : If three 0-simplexes $u,v,w$ spans a 2-simplex, the path $...uvw...$ can be replaced by the path $...uw...$ , and vice-versa (Fig. \ref{mvt:simp1}). If $v$ and $w$ are 0-simplexes of $T$ then the path $...vwv...$ can be replaced by the path $...v...$ and vice-versa (Fig. \ref{mvt:simp2}).

\begin{figure}[h!]
$$
\vcenter{\hbox{\includegraphics[height=3cm]{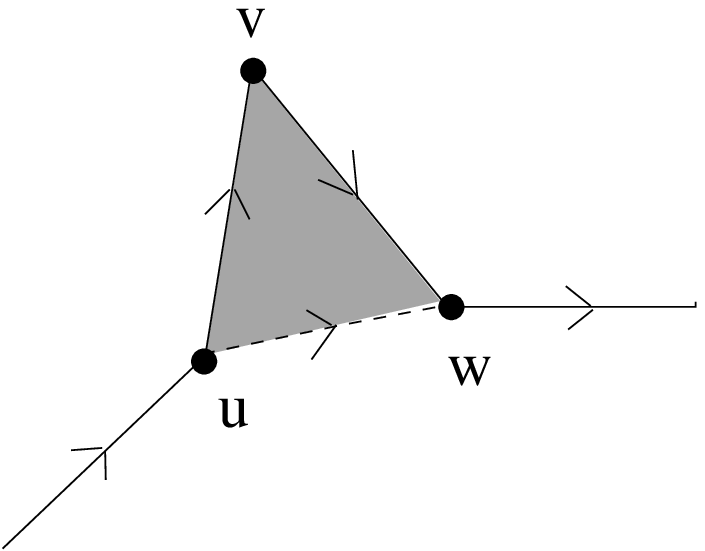}}}\hspace{1cm}\longleftrightarrow\hspace{1cm}\vcenter{\hbox{\includegraphics[height=3cm]{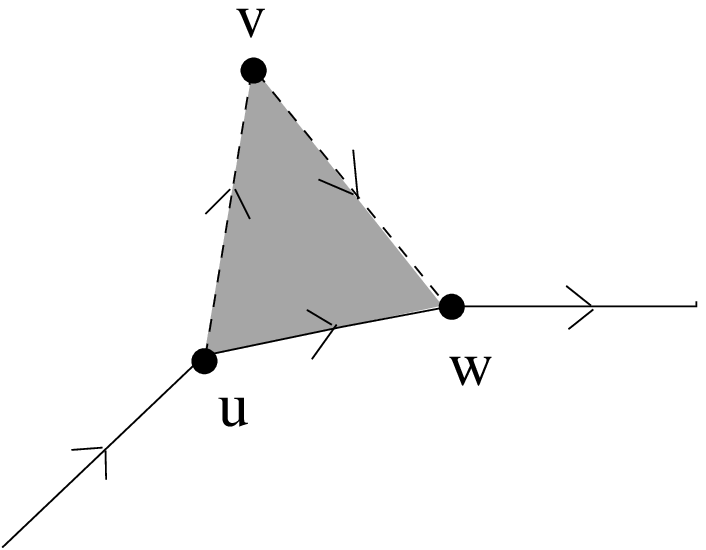}}}
$$
\caption{}
\label{mvt:simp1}
\end{figure}

\begin{figure}[h!]
$$
\vcenter{\hbox{\includegraphics[height=3cm]{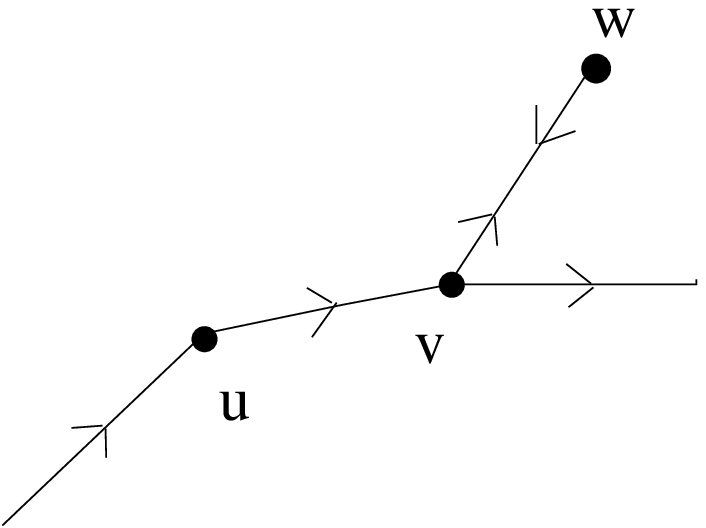}}}\hspace{1cm}\longleftrightarrow\hspace{1cm}\vcenter{\hbox{\includegraphics[height=3cm]{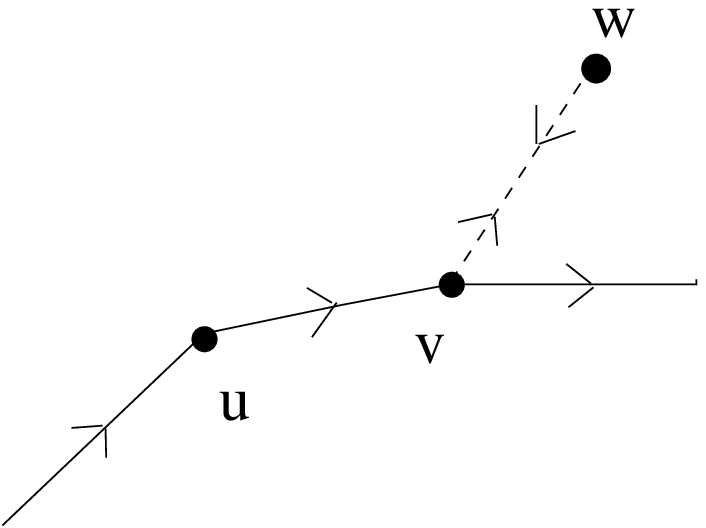}}}
$$
\caption{}
\label{mvt:simp2}
\end{figure}
These moves define an equivalence relation over the set of paths, \emph{two paths are equivalents} if and only if we can go between any of these two paths by a finite sequence of moves (Fig \ref{mvt:simp1}) and/or (Fig. \ref{mvt:simp2}). We will abusively use the notation $v_1...v_n$ for the equivalence class of the path $v_1...v_n$.

Let $T$ be a simplicial complex, the \emph{fundamental groupoid} of $T$ is the category with objects the 0-simplexes of $T$ and morphisms the equivalence classes of paths. The composition is given by the concatenation of paths. For technical reason, we will consider the opposite category (morphisms are reversed), more precisely the composition of oriented 1-simplexes :
$$
...\fd x\xr{(xy)} y \xr{(yz)} z \fd ...
$$
will be written : $(yz)\circ(xy)=(xy)(yz)$. We denote by $\pi_1(T)$ the fundamental groupoid of $T$. Let $v$ be a 0-simplex of $T$, we denote by $\pi_1(T,v)$ the category with one object $v$ and whose morphisms are equivalence classes of $v$-loops. There is an equivalence of categories between the categories $\pi_1(T)$ (resp. $\pi_1(T,v)$) and  the fundamental group of the topological space $\cs{T}$ obtained from $T$ (resp. the fundamental group of the pointed topological space $(\cs{T},v)$).

A \emph{connected simplicial complex} is a simplicial complex $T$ such that for every 0-simplexes $u$ and $v$ there exists a path between $u$ and $v$.

\begin{lem}
Let $T$ be a simplicial complex and $v$ be a 0-simplex of $T$. There is an equivalence of categories between $\pi_1(T)$ and $\pi_1(T,v)$.
\end{lem}
\pf
The inclusion functor :
\begin{align*}
\pi_1(T,v) & \fd \pi_1(T) \\
v & \ap v\\
vv_1....v_nv & \ap vv_1...v_nv\, ,
\end{align*}
is a faithful functor. Since $T$ is connected, this functor is essentially surjective.

\qed

\subsection{Colorings}
We give a topological interpretation of the set of colorings. Throughout this section $\C$ will be a finitely semisimple tensor $\Bbbk$-category and $G$ will be a group.

\subsubsection{Description of the set of colorings in the case of manifolds without boundary}

Let $T$ be a simplicial complex. A \emph{$G$-coloring $c$} of $T$ is a map :

\begin{align*}
c:T^1_o & \fd G\\
e & \ap c(e),
\end{align*}
satisfying :
\begin{itemize}
\item[(i)]for every  oriented 1-simplex $(x_1x_2)$ of $T$~:~$c(x_1x_2)=c(x_2x_1)^{-1}$\, ,
\item[(ii)]for every oriented 2-simplex  $(x_1x_2x_3)$ of $T$~:~$c(x_1x_2)c(x_2x_3)c(x_3x_1)=1$\, .\\
\end{itemize}

We denote by $Col_G(T)$ the set of $G$-colorings.

\emph{A gauge of $T$} is a map $\delta : T^0\fd G$. The \emph{gauge group of $T$} is the group of gauges of $T$ and is denoted by $\m{G}_T$. The gauge group $\m{G}_T$ acts on $Col_G(T)$ in the following way :
\begin{align}
\m{G}_T \times Col_G(T) & \fd Col_G(T) \label{action:jauge}\\
(\delta,c) & \ap c^{\delta},\nonumber
\end{align}
where $c^{\delta}$ is the coloring : $c^{\delta}(xy)=\delta(x)c(xy)\delta(y)^{-1}$, for every oriented 1-simplex $(xy)$. We denote by $\jag{T}$ the quotient set of $Col_G(T)$ by the action of the gauge group $\m{G}_T$.

Let $c$ be a $G$-coloring of $T$, we denote by $[c]$ the class of $c$ in $\jag{T}$. For every group $G$, we denote by $G$ the groupoid with one object and whose the set of morphisms is $G$. For easy reading, we will consider the opposite groupoid, thus the composition of map $g$ with $f$ will be written $fg$.

\begin{pro}\label{pro:jau}
Let $T$ be a simplicial complex, $\C$ be a semisimple tensor $\Bbbk$-category and $G$ be a group. The map :
\begin{align*}
Col_G(T) & \fd Fun(\pi_1(T),G)\\
c& \ap F_c \, ,
\end{align*}
where $F_c$ is the functor which sends every 0-simplex of $T$ to the unique object of $G$ and sends every oriented 1-simplex $(xy)$ to $c(xy)$, induces the following isomorphism :
\begin{equation}
\jag{T}\cong Fun(\pi_1(T),G)/(iso)\cong [\cs{T},K(G,1)]\, ,\label{pro:suitebij}
\end{equation}
where $[\cs{T},K(G,1)]$ is the set of homotopy classes of continuous maps from the topological space $\cs{T}$ to the Eilenberg-Mac Lane space $K(G,1)$.
\end{pro}

\pf
For every coloring $c\in Col_G(T)$, we define the following functor $F_c : \pi_1(T) \fd G$, with $F_c(x)=\star$ for every 0-simplex $x$ and $F_c(xy)=c(xy)$ for every oriented 1-simplex $(xy)$. The composition of morphisms is given by concatenation, it follows that : $F_c(x_1x_2...x_{n-1}x_n)=F_c(x_1x_2)...F_c(x_{n-1}x_n)$. For every oriented 1-simplex $(xy)$, we have $c(xy)=c(yx)^{-1}$ and for every oriented 2-simplex $(xyz)$, we have : $c(xy)c(yz)=c(xz)$. It follows that $F_c$ is well defined on the equivalent classes of paths and thus it is a functor from $\pi_1(T)$ to $G$. We have build the following map :
\begin{align}
\Psi : Col_G(T) & \fd Fun(\pi_1(T),G) \label{bij:homotopie}\\
c & \ap F_c, \nonumber
\end{align}
Let us show that this map is bijective. If two colorings $c$ and $c'$ define the same functor then for every oriented 1-simplex $(xy)$ of $T$, we have~: $c(xy)=F_c(xy)=F_{c'}(xy)=c'(xy)$. Let $F$ be a functor from $\pi_1(T)$ to $G$. For every oriented 1-simplex $(xy)$ of $T$, we set : $c(xy)=F(xy)$. For every oriented 2-simplex $(xyz)$ of $T$, we have : $F(xy)F(yz)=F(xz)$ and for every oriented 1-simplex $(xy)$ we have $F(xy)F(yx)=F(xx)=\id_{\star}$, thus $c\in Col_G(T)$.

Let us show that the bijection (\ref{bij:homotopie}) induces a bijection from $Col_{G}(T)/\m{G}_T$ to $Fun(\pi_1(T),G)/(iso)$. Let $c$ and $c'$ be two $G$-colorings of $T$ such that $c'=c^{\delta}$ with $\delta\in \m{G}_T$, we set $F=\psi(c)$ and $F'=\Psi(c')$. For every object $x$ of $\pi_1(T)$, we set $\eta_x=\delta(x) : F'(x) \fd F(x)$. Let us show that $\eta$ is a natural isomorphism between $F$ and $F'$. It suffices to check out for every oriented 1-simplex. For every oriented 1-simplex $(xy)$ of $T$, we have~:
\begin{align*}
F'(xy)\eta_y&=c^{\delta}(xy)\delta(y)\\
&=\delta(x)c(xy)\\
&=\eta_xF(xy)\, .
\end{align*}
 Thus $\Psi$ induces a bijection from $Col_G(T)/\m{G}_T$ to $Fun(\pi_1(T),G)/(iso)$.

\qed

\subsubsection{Description of the set of colorings in  the case of manifolds with boundary}

Let $G$ be a group, $M$ be a 3-manifold, $\Sigma$ be the boundary of $M$ and $T_0$ be a triangulation of $\Sigma$. Set $Col_{G,c_0}(T)$ the set of $G$-colorings of $T$ such that the restriction to $T_0$ is $c_0$. For every functor $F_0 : \pi_1(T_0)\fd G$, $Fun(\pi_1(T),G)_{F_0}$ is the set of functors $F$ from $\pi_1(T)$ to the groupoid $G$ such that the diagram :
$$
\xymatrix{
\pi_1(T) \ar[r]^{F}&  G \\
\pi_1(T_0) \ar@{^{(}->}[u]^{i} \ar[ur]_{F_0}\, ,&&
}
$$
commutes. In the above diagram, $i$ is the inclusion functor. We denote by $Fun(\pi_1(T),G)_{F_0}/(iso)$ the set of isomorphisms classes of functor in $Fun(\pi_1(T),G)_{F_0}$  such that the restriction of the natural isomorphisms to $\pi_1(T_0)$  is $\id_{F_0}$.

\begin{pro}\label{pro:colbound}
Let $\C$ be a semisimple tensor $\Bbbk$-category, $T$ be a simplicial complex and $T_0$ a subcomplex of $T$. For every coloring $c_0\in Col(T_0)$, the map:
\begin{align}
Col_{G,c_0}(T)&\fd Fun(\pi_1(T),G)_{F_{c_0}} \label{pro:tec}\\
c & \ap F_c,\nonumber
\end{align}
where the functor $F_c$ sends every 0-simplex of $T$ to the unique object of the groupoid $G$ and every oriented 1-simplex $(xy)$ to $c(xy)$, induces the following isomorphism :
\begin{equation}
Col_{G,c_0}(T)/\m{G}_T\simeq Fun(\pi_1(T),G)_{F_{c_0}}/(iso)\, .
\end{equation}
\end{pro}
\pf

The proof is the same as that of proposition \ref{pro:jau}, i.e. we show that the functor $F_c$ is well defined and the map (\ref{pro:tec}) induces an isomorphism between the quotient spaces.

\qed

\subsection{Construction of the homotopical Turaev-Viro invariant}

From now on, every spherical category has an invertible dimension.

\subsubsection{Notations}

Let $\C$ be a spherical category. The graduator $\grad$ of $\C$ is a finite group. In this case the Eilenberg-Mac Lane space $K(\grad,1)$ is the classifying space $B\grad$. From now on, we will use the notation $B\grad$ and the terminology \emph{classifying space}. Let $M$ be a 3-manifold and $T$ be a triangulation of $M$, for every $x\in [M,B\grad]$, we denote by $Col_x(T)$ the set of colorings $c$ of $T$ such that the equivalence class $[c]$ in $Col_{\grad}(T)/\m{G}_T$ corresponds to $x$. We obtain a partition of the set $Col(T)$~: $\dis{Col(T)=\coprod_{x\in [M,B\grad]}Col_x(T)}$. If $c\in Col(T)$, we denote by $x_c\in [M,B\grad]$ the homotopy class associated to $c$ by the bijection (\ref{pro:suitebij}).

Let $M$ be a manifold,  $\Sigma$ be the boundary of $M$ and  $T_0$ be a triangulation of $\Sigma$. For every homotopy class $x_0\in [\Sigma,B\grad]$, we denote by $[M,B\grad]_{\Sigma,x_0}$ the set of homotopy classes of maps from $M$ to the classifying space $B\grad$ such that the homotopy class of the restriction to $\Sigma$ is $x_0$. Thus for every coloring $c_0\in Col(T_0)$ and for every triangulation $T$ of $M$ such that the restriction to $\Sigma$ is $T_0$, we have the isomorphisms :
\begin{equation}\label{isobord}
Col_{\grad,c_0}/(\m{G}_T) \cong \dis{Fun(\pi_1(T),\grad)_{F_{c_0}}/(iso)  \cong [M,B\grad]_{\Sigma,x_{c_0}}}\, .
\end{equation}

For every coloring $c_0\in Col(T_0)$ and for every homotopy class $y\in [M,B\grad]_{\Sigma,x_{c_0}}$, we denote by $Col_{c_0,y}(T)$ the set of colorings $c\in Col(T)$ satisfying :
\begin{itemize}
\item $c_{T_0}=c_0$,
\item  the equivalent class $[c]\in Col_{\grad,c_0}/\m{G}_T$ corresponds to  $y\in [M,B\grad]_{\Sigma,x_{c_0}}$ by the bijections (\ref{isobord}).
\end{itemize}

Let $\C$ be a spherical category, $M$ be a 3-manifold, $\Sigma$ be the boundary of $M$, $T_0$ be a triangulation of $\Sigma$ and $c_0\in Col(T_0)$. We can break up the Turaev-Viro state sum in the following way~:
\begin{align*}
TV_{\C}(M,c_0)&= \dc^{-n_0(T)+n_0(T_0)/2}\sum_{c \in Col_{c_0}(T)} w_c W_c \\
&= \dc^{-n_0(T)+n_0(T_0)/2} \sum_{x\in [M,B\grad]_{(\Sigma,x_{c_0})}}\sum_{c\in Col_{c_0,x}(T)}w_cW_c\, ,
\end{align*}
set : $\dis{HTV_{\C}(M,x,c_0)=\dc^{-n_0(T)+n_0(T_0)/2}\sum_{c\in Col_{c_0,x}}w_c W_c}$. Let us show that for every coloring $c_0\in Col(T_0)$, $HTV_{\C}(M,x,c_0)$ is an invariant for the triple $(M,x,c_0)$. To prove that we will show that for a fixed triangulation of the boundary of $M$ and a fixed coloring $c_0$ of the boundary $HTV(M,x,c_0)$ is invariant under the Pachner moves \cite{Pachner}. We recall briefly the Pachner theorem~:
\begin{theorem}[Pachner theorem \cite{Tu}]
Two triangulations of a compact 3-manifold which coincide on the boundary of $M$ can change one into the other by a finite sequence of ambient isotopy and/or the following local moves~:
\begin{figure}[h!]
\begin{tabularx}{11.5cm}{cXc}
$\vcenter{\hbox{\includegraphics[height=1.5cm,width=1.5cm]{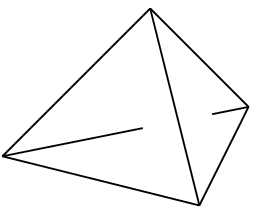}}}\hspace{.5cm}\longleftrightarrow\hspace{.5cm}\vcenter{\hbox{\includegraphics[height=1.5cm,width=1.5cm]{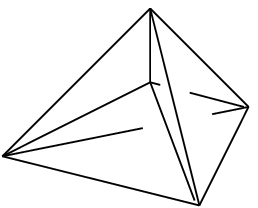}}}$
 & \quad \quad \quad & $\vcenter{\hbox{\includegraphics[height=1.5cm,width=1.5cm]{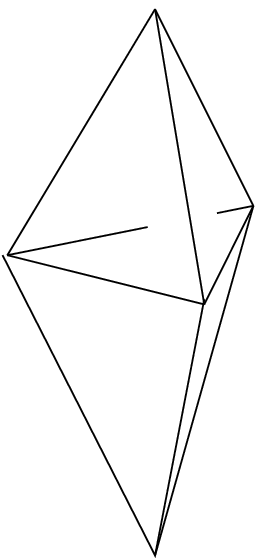}}}\hspace{.5cm}\longleftrightarrow\hspace{.5cm}\vcenter{\hbox{\includegraphics[height=1.5cm,width=1.5cm]{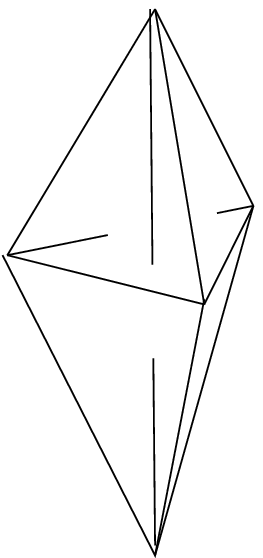}}}$
\\
& & \\
Pachner move  (1-4)
\label{des:P14} & \quad \quad \quad&
Pachner move  (2-3)
\label{des:P23}
\end{tabularx}
\end{figure}
\end{theorem}

First, we will show that for every $c_0\in Col(T_0)$ the set $Col_{\grad,c_0}(T)/\m{G}_T$ is invariant under the Pachner moves.

 Let $T$ be a simplicial complex and $T_0$ be a subcomplex of $T$, we denote by $T_1$ the simplicial complex obtained from $T$ by substituting a 3-simplex $\vcenter{\hbox{\includegraphics[height=1cm,width=1cm]{figs/tetra1.eps}}}\hspace{.5cm}$ (non contained in $T_0$) with $\vcenter{\hbox{\includegraphics[height=1cm,width=1cm]{figs/tetra4.eps}}}\hspace{.5cm}$ (Pachner move 1-4). We denote by $T_2$ the simplicial complex obtained from $T$ by substituting $\vcenter{\hbox{\includegraphics[height=1cm,width=1cm]{figs/tetra2.eps}}}\hspace{.5cm}$ (non contained in $T_0$) with $\vcenter{\hbox{\includegraphics[height=1cm,width=1cm]{figs/tetra3.eps}}}\hspace{.5cm}$ (Pachner move 2-3). For every coloring $c\in Col(T_1)$ (resp. $Col(T_2)$), we denote by $c_T$ the restriction of $c$ to the simplicial complex $T$.

\begin{lem}\label{lem:invjau}
Let $T$ be a simplicial complex, $T_0$ be a simplicial subcomplex of $T$ and $c_0\in Col(T_0)$, the following maps :
\begin{align}
Col_{\grad,c_0}(T_1)/\m{G}_{T_1} & \fd Col_{\grad,c_0}(T)/\m{G}_{T}\label{T1}\\
[c]& \ap [c_T] \nonumber ,
\end{align}

\begin{align}
Col_{\grad,c_0}(T_2)/\m{G}_{T_2} & \fd Col_{\grad,c_0}(T)/\m{G}_{T}\label{T2}\\
[c] & \ap [c_T] \nonumber.
\end{align}

are bijective.
\end{lem}
\pf Let us show that the map (\ref{T1}) is bijective. Set :
\begin{align*}
\phi_1 : Col_{\grad,c_0}(T_1) &\fd Col_{\grad,c_0}(T) \\
c & \ap c_T\, .
\end{align*}
This map is surjective.

Let $c,c'\in Col_{\grad,c_0}(T_1)$ such that $\phi_1(c)=\phi_1(c')$, we denote by $i$ the 0-simplex inside the simplicial complex $\vcenter{\hbox{\includegraphics[height=1cm,width=1cm]{figs/tetra4.eps}}}\hspace{.5cm}$ and the other 0-simplexes are denoted by an integer from 1 to 4. Set $\delta : T_1^0 \fd \grad$ such that $\delta(x)=\left\{\begin{array}{cc} c'(1i)^{-1}c(1i) & \mbox{if $x=i$} \\ \Bbbi & \mbox{otherwise}\end{array}\right.$. Thus for every oriented 1-simplex $(ki)$, where $k\in \{2,3,4\}$ :
\begin{align*}
c^{\delta}(ki)&=c(ki)\delta^{-1}(i) \\
&=c(k1)c(1i)\delta^{-1}(i) \\
&= c(k1)c(1i)c(1i)^{-1}c'(1i) \\
&= c'(k1)c'(1i) \\
&= c'(ki),
\end{align*}
moreover $c^{\delta}(1i)=c(1i)\delta(i)^{-1}=c'(1i)$. Thus $\phi_1$ defines a map $\ov{\phi_1}$ from the quotient space $Col_{\grad,c_0}(T_1)/\m{G}_{T_1}$ to the quotient space $Col_{\grad,c_0}(T)/\m{G}_T$. This map is well defined since if $c=c'^{\delta}$ then $c_T={c'_T}^{\delta}$ and is surjective. Let us show that $\ov{\phi_1}$ is injective. Let $c,c'\in Col_{\grad,c_0}(T_1)$, if $\ov{\phi_1([c])}=\ov{\phi_1([c'])}$ then $[c_T]=[c'_T]$ thus there exists $\delta\in \m{G}_T$  such that : $c_T={c'_T}^{\delta}$. Set : $\delta'(x)=\left\{\begin{array}{cc} c(1i)^{-1}\delta(1)c'(1i) & \mbox{ if $x=i$}\\ \delta(x) & \mbox{otherwise} \end{array}\right.$. Thus  we have :
\begin{align*}
c'^{\delta'}(1i)&= \delta'(1)c'(1i)\delta'(i)^{-1}\\
&=\delta(1)c'(1i)c'(1i)^{-1}\delta(1)^{-1}c(1i)\\
&=c(1i)
\end{align*}
and for every $k\in \{2,3,4\}$ :
\begin{align*}
c'^{\delta'}(ki)&= c'^{\delta'}(k1)c'^{\delta'}(1i)\\
&= c'^{\delta'}(k1)c(1i)\\
&= c'^{\delta}(k1)c(1i)\\
&= c(k1)c(1i)\\
&= c(ki)\, .
\end{align*}
Thus $\ov{\phi_1}$ is a bijection from $Col_{\grad,c_0}(T_1)/\m{G}_{T_1}$ to $Col_{\grad,c_0}(T)/ \m{G}_T$.

Let us show that the map (\ref{T2}) is a bijection. Set :
\begin{align*}
\phi_2 : Col_{\grad,c_0}(T_2) &\fd Col_{\grad,c_0}(T) \\
c & \ap c_T.
\end{align*}

Let $c,c'\in Col_{\grad}(T_2)$ such that $\phi_2(c)=\phi_2(c')$, then the colorings $c$ and $c'$ are equals on the oriented 1-simplexes of  $T$. Let us show that the equality is still true on the remaining 1-simplex. We set the following numbering : $$
\begin{psfrags}
\psfrag{i}{$i$}
\psfrag{j}{$j$}
\psfrag{k}{$k$}
\psfrag{l}{$l$}
\psfrag{m}{$m$}
\vcenter{\hbox{\includegraphics[height=2cm,width=2cm]{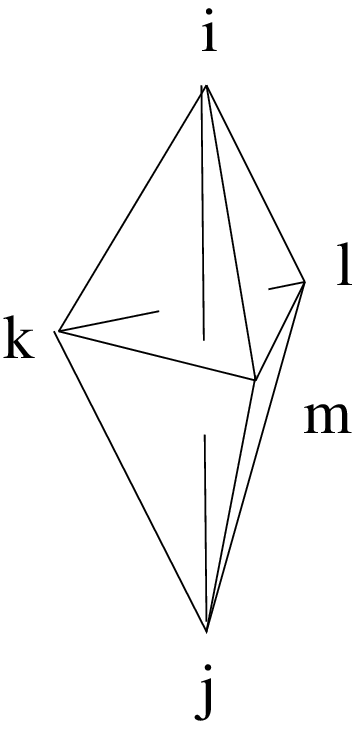}}}.
\end{psfrags}
$$
It follows : $c(ij)=c(ik)c(kj)=c'(ik)c'(kj)=c'(ij)$. Thus the map $\phi_2$ is injective. Let $c$ be a coloring of $T$ and set the following numbering :
$$
\begin{psfrags}
\psfrag{i}{$i$}
\psfrag{j}{$j$}
\psfrag{k}{$k$}
\psfrag{l}{$l$}
\psfrag{m}{$m$}
\vcenter{\hbox{\includegraphics[height=2cm,width=2cm]{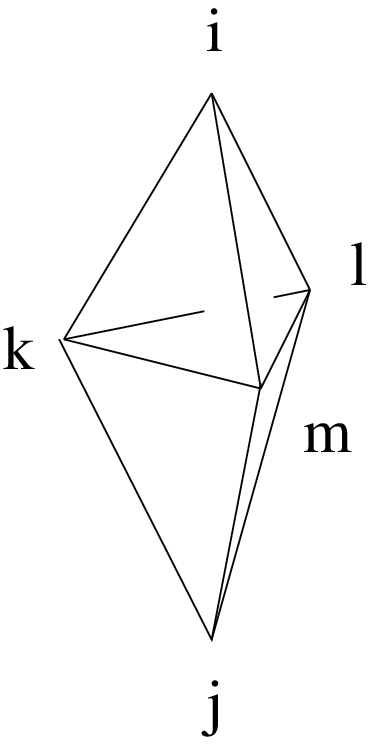}}}.
\end{psfrags}
$$
Set : $c'(e)=\left\{\begin{array}{cc}  c(ik)c(kj) & \mbox{if $e = (ij)$}\\c(e) & \mbox{otherwise}\\\end{array}\right.$, this a coloring of $T_2$, indeed :
\begin{align*}
c'(il)c'(lj)&=c(il)c(lj) \\
&= c(ik)c(kl)c(lj) \\
&=c'(ij)
\end{align*}
and : $c'(im)c'(mj)=c'(ij)$. Furthermore we have $\phi_2(c')=c$, thus we have a bijection between $Col_{\grad,c_0}(T)$ and $Col_{\grad,c_0}(T_2)$. It induces the bijection (\ref{T2}).

\qed

\begin{theorem}\label{thm:inv}
Let $\C$ be a spherical, $M$ be 3-manifold, $\Sigma$ be the boundary of $M$ and $T_0$ be a triangulation of $\Sigma$. For every coloring $c_0\in Col(T_0)$ and for every homotopy class $x\in [M,B\grad]_{\Sigma,x_{c_0}}$, where $x_{c_0}\in [\Sigma,B\grad]$ is obtained from $c_0$, the vector :
$$HTV_{\C}(M,c_0,x)=\dc^{-n_0(T)+n_0(T_0)/2}\sum_{c\in Col_{c_0,x}(T)}w_cW_c \in V_{\C}(\Sigma,T_0,c_0)\,$$
is an invariant of the triple $(M,x,c_0)$. We have the following equality~:
\begin{equation}\label{TVsplit}
TV_{\C}(M,c_0)=\sum_{x\in [M,B\grad]_{\Sigma,x_{c_0}}}HTV_{\C}(M,c_0,x)\, .
\end{equation}
\end{theorem}
\pf
The splitting (\ref{TVsplit}) comes from the Turaev-Viro state-stum and the partition of the set of colorings. Let us show that $HTV_{\C}(M,c_0,x)$ is an invariant under the Pachner move (1-4). As previously, we denote by $T_1$ the simplicial complex obtained from $T$ by substituting $\vcenter{\hbox{\includegraphics[height=1cm,width=1cm]{figs/tetra1.eps}}}\hspace{.5cm}$ with $\vcenter{\hbox{\includegraphics[height=1cm,width=1cm]{figs/tetra4.eps}}}\hspace{.5cm}$. To show the invariance under the Pachner move (1-4), we must show  the following equality :
$${w_{c}W_{c}=\dc^{-1}\sum_{\substack{c'\in Col_{c_0,x}(T_1)\\ c'_{T}=c}}w_{c'}W_{c'}}\, ,$$
for every coloring $c\in Col_{c_0,x}(T)$. By construction of the Turaev -Viro invariant \cite{Tu}, for every coloring $c\in Col_{c_0,x}(T)$ we have~:~$\dis{w_{c}W_{c}=\dc^{-1}\sum_{\substack{c'\in Col_{c_0}(T_1)\\ c'_{T}=c}}w_{c'}W_{c'}}\, .$ Using the bijection (\ref{T1}), we know that for every coloring $c\in Col(T)_{c_0,x}$, if $c'\in Col_{c_0}(T_1)$ and $c'_T=c$ then $c'\in Col_{c_0,x}(T_1)$. It results the invariance under the Pachner move (1-4).

Let us show the invariance under the Pachner move (2-3). As previously, we denote by $T_2$  the simplicial complex obtained from $T$ by substituting $\vcenter{\hbox{\includegraphics[height=1cm,width=1cm]{figs/tetra2.eps}}}\hspace{.5cm}$ with $\vcenter{\hbox{\includegraphics[height=1cm,width=1cm]{figs/tetra3.eps}}}\hspace{.5cm}$. To show the invariance under the Pachner move (2-3), we must show the following equality : $\dis{w_{c}W_{c}=\sum_{\substack{c'\in Col_{c_0,x}(T_2)\\ c'_{T}=c}}w_{c'}W_{c'}}\, ,$ for every coloring $c\in Col_{c_0,x}(T)$. Let $c\in Col_{c_0,x}(T)$, by construction of the Turaev-Viro invariant we have :
$${w_{c}W_{c}=\sum_{\substack{c'\in Col_{c_0}(T_2)\\ c'_{T}=c}}w_{c'}W_{c'}}\, ,$$
 using the bijection (\ref{T2}), we know that if $c\in Col_{c_0,x}(T)$ then for every coloring $c'\in Col_{c_0}(T_2)$ such that $c'_T=c$, we have : $c'\in Col_{c_0,x}(T_2)$. It follows the invariance under the Pachner move (2-3).

\qed

\begin{cor}\label{cor:split}
Let $\C$ be a spherical category and $M$ be a 3-manifold without boundary, the splitting of the Turaev-Viro invariant (\ref{TVsplit}) is~:
\begin{equation}\label{TVsplitwb}
TV_{\C}(M)=\sum_{x\in [M,B\grad]}HTV_{\C}(M,x)\, .
\end{equation}
\end{cor}

\medskip

The invariant $HTV_{\C}$ is called \emph{the homotopical Turaev-Viro invariant}.

\medskip

We can extend the Turaev-Viro and the homotopical Turaev-Viro invariant to singular triangulated manifolds. Indeed there is a Pachner theorem for singular triangulated manifolds~:
\begin{theorem}[\cite{BW}]
Two triangulated singular 3-manifolds are piecewise-linear homeomorphic if and only if they are related by a finite sequence of ambient isotopy and/or Pachner moves (1-4) and/or (2-3).
\end{theorem}
Thus we can defined the Turaev-Viro invariant and the homotopical Turaev-
Viro invariant to singular triangulated manifold. In this case we obtain the same splitting (\ref{TVsplit}) and (\ref{TVsplitwb}).

\section{Examples}\label{sec:calcul}

\subsection{The three dimensional sphere $S^3$}

The three dimensional sphere $S^3$ admits a singular triangulation with one tetrahedron \cite{Jacorubin}, denoted by $(x_1,x_2,x_3,x_4)$ , and the following identifications of the 2-simplexes~:
\begin{align*}
(x_1,x_4,x_2)&=(x_3,x_4,x_2) \, ,\\
(x_1,x_3,x_4)&=(x_3,x_2,x_1) \, .
\end{align*}
After this identification, we obtain a singular triangulation with one 0-0-simplex denoted by $x$ and two 1-simplexes denoted by $a$ and $b$. Below is a figure of this singular triangulation.

\[
\xy
(0,0)*+{x}="A"; (15,-5)*+{x}="B";
(25,5)*+{x}="C";(8,15)*+{x}="D";
"B";"A" **\dir{-} ?(.5)*\dir{>}+(-1,-2)*{\scriptstyle a}; "A";"D" **\dir{-} ?(.5)*\dir{>}+(-2,0)*{\scriptstyle a} ; "B";"C" **\dir{-} ?(.5)*\dir{>}+(1,-1)*{\scriptstyle a}; "D";"B" **\dir{-} ?(.5)*\dir{>>}+(-2.5,2)*{\scriptstyle b}; "C";"D" **\dir{-} ?(.5)*\dir{>}+(2,1)*{\scriptstyle a} ; "C";"A" **\dir{--} ?(.5)*\dir{>}+(1.5,2)*{\scriptstyle a};
(13,-10)*{S^3};
\endxy
\]

\subsubsection{Group categories}

\subsubsection*{{\bf Colorings.}}

Let $G$ be a finite group, $\C$ be a group category such that $G$ is the group of scalar objects of $\C$ (up to isomorphism) and $c$ be a coloring of $S^3$. With the above notations, we set $c(a)=g$ and $c(b)=h$, by construction of the singular triangulation of $S^3$, we obtain the following relations~:
\begin{align*}
g^2&=g\, ,\\
g^2h&=1\, .
\end{align*}
It results that for every coloring $c$, we have~: $c(a)=c(b)=1$. Reciprocally the data $(1,1)\in G^2$ defines a coloring of $S^3$. Since there is only one coloring, the homotopical Turaev-Viro invariant is equal to the Turaev-Viro invariant~:
\begin{equation*}
TV_{\C}(S^3)=HTV_{\C}(S^3,0)=\frac{1}{\sharp G}\, ,
\end{equation*}
where $0\in [S^3,BG]$ is the trivial homotopy class.
\subsubsection{$U_q(\mathfrak{sl}_2)$ with $q$ root of unity}
Let $r\geq 3$ and $A$ be a primitive $2r$-th root of unity such that $A^2=q$ is a primitive $r$-th  root of unity. The set of scalar objects (up to isomorphism) is given by the set of integers $\{0,...,r-2\}$. The graduator of this category is the cyclic group $\Zz_2$. Throughout this section we will consider the group $\Zz_2$ endowed with the  multiplicative notation. The dimension of the  scalar object $i$ is : $\dim(i)=(-1)^i\frac{A^{i+1} - A^{-i-1}}{A - A^{-1}}=(-1)^i[i+1]_q$, with $0\leq i\leq r-2$. The dimension of the category is the scalar~:~$\dis{\Delta_{U_q(\frak{sl}_2)}=\sum_{i=0}^{r-2}\dim(i)^2=\frac{-2r}{(A-A^{-1})^2}}$.

\subsubsection*{{\bf Colorings}}

For every coloring $c$ of the above singular triangulation of $S^3$, we set : $X=c(a)$ and $Y=c(b)$. By construction of the singular triangulation of $S^3$, one gets~:
\begin{enumerate}
\item[(i)] $3X\leq 2r-4$ and $2X+Y\leq 2r-4$,
\item[(ii)] $Y\leq 2X$,
\item[(iii)] $X=0$ mod 2 and $Y=0$ mod 2\, .\\
\end{enumerate}
Conversely a pair of positive integers $(X,Y)$, with $0\leq X,Y\leq r-2$, which respects the conditions $(i)$, $(ii)$ and $(iii)$ defined a coloring $c$ of $S^3$, in such a way~: $c(a)=X$ and $c(b)=Y$. From now on, we denote $c=(X,Y)$ a coloring of $S^3$.

Let $c=(X,Y)$ be a coloring of $S^3$, we denote by $\cs{X}$ the image of the scalar object $X$ in the graduator $\Zz_2$. The relation $(iii)$ implies~: $\cs{X}=\cs{Y}=1$, thus there is a unique equivalence class of coloring. It follows that the Turaev-Viro  invariant is equal to the homotopical Turaev-Viro invariant~:

\begin{equation}\label{TVS3}
  TV(S^3)=HTV(S^3,0)=\frac{1}{\Delta}\, ,
\end{equation}
where $0\in [S^3,B\Zz_2]$ is the trivial homotopy class.

\subsection{The 3-torus $\torust$}

Below is a singular triangulation of the 3-torus $\torust$~:

\[
\xy (0,0)*+{x}="A"; (30,0)*+{x}="B";
(0,40)*+{x}="C";(30,40)*+{x}="D";(20,50)*+{x}="E";(50,50)*+{x}="F";(50,10)*+{x}="G";(20,10)*+{x}="H";
  "A";"B" **\dir{-} ?(.5)*\dir{>}+(0,-3)*{\scriptstyle a};  "A";
"C" **\dir{-} ?(.5)*\dir{>>}+(-3,0)*{\scriptstyle c};  "C"; "D" **\dir{-} ?(.5)*\dir{>}+(2,-1)*{\scriptstyle a} ;  "B";"D" **\dir{-} ?(.5)*\dir{>>}+(1,0)*{\scriptstyle c} ; "C";"E"**\dir{-} ?(.5)*\dir{)}+(-3,0)*{\scriptstyle b}; "C";"F" **\dir{-} ?(.5)*\dir{|-}+(1,-1)*{\scriptstyle f} ; "E";"F"**\dir{-} ?(.5)*\dir{>}+(0,2)*{\scriptstyle a}; "D";"F"**\dir{-} ?(.5)*\dir{)}+(-1,-2)*{\scriptstyle b};  "G";"F"**\dir{-} ?(.5)*\dir{>>}+(2,0)*{\scriptstyle c};"B";"G"**\dir{-} ?(.5)*\dir{)}+(2,-1)*{\scriptstyle b};"H";"G" **\dir{--} ?(.5)*\dir{>}+(-3,1)*{\scriptstyle a};"A";"G" **\dir{--} ?(.5)*\dir{|-}+(2,-1)*{\scriptstyle f};"H";"E"
**\dir{--} ?(.5)*\dir{>>}+(1,2)*{\scriptstyle c};"A";"E" **\dir{--} ?(.5)*\dir{>>|}+(2,0)*{\scriptstyle d};"B";"F" **\dir{-} ?(.5)*\dir{>>|}+(2,0)*{\scriptstyle d};"A";"D"**\dir{-} ?(.5)*\dir{>|}+(1.75,0)*{\scriptstyle e};"F";"A" **\dir{--} ?(.5)*\dir{}+(0,2)*{\scriptstyle g};"A";"H" **\dir{--}?(.5)*\dir{)}+(0,2)*{\scriptstyle b};"H";"F"
**\dir{--}?(.5)*\dir{>|}+(2,0)*{\scriptstyle e};
(25,-8)*{\torust}
\endxy
\]

\subsubsection{Group categories.}

\subsubsection*{{\bf Colorings.}}

Let $G$ be a finite group, $\alpha\in H^3(G,\Bbbk^*)$. Let $\tilde{c}$ be a coloring of the above singular triangulation of $\torust$, it follows that $\tilde{c}$ is defined by the 7-tuple $(a,b,c,d,e,f,g,h)$. Furthermore by definition of the colorings, we have the following relations~:
\begin{align*}
d&= bc=cb\, ,  \\
e&= ac = ca \, , \\
f&= ab =ba \, ,\\
g &= fc =abc \, .
\end{align*}
It follows that the coloring $\tilde{c}$ is described by the triple $(a,b,c)$, where $ab=ba$, $ac=ca$ and $bc=cb$. Reciprocally, every triple $(a,b,c)$ such that $ab=ba$, $ac=ca$ and $bc=cb$ defined a coloring of $\torust$. For every $g\in G$, we denote by $N_g$ the set $\{h\in G | hg=gh\}$. The set of coloring of the 3-torus for the above singular triangulation is $\{(a,b,c)\in G^3|b\in N_g \, ,l c\in N_g \mbox{ and } bc=cb\}$. From now on, every coloring $\tilde{c}$ of the 3-torus will be denoted by $\tilde{c}=(a,b,c)$, with $a,b,c\in G^3$ such that $b,c\in N_g$ and $bc=cb$.

\subsubsection{{\bf Gauge actions.}}

We have built a  singular triangulation with one 0-0-simplex, thus the gauge group on this singular triangulation can be identified to the group $G$. Since the graduator of a group category is the group of scalar objects (up to isomorphism). Thus the colorings have value in the graduator and so the gauges act directly on the colorings. Let us describe the equivalent classes of colorings under the gauge action.

 Let $\tilde{c}=(a,b,c)$ and $\tilde{c}'=(a',b',c')$  be two colorings of the 3-torus, they are equivalent under the gauge action if and only if there exists a gauge $h\in G$,  such that~:
\begin{align}
a' &= hah^{-1}\, ,\\
b' &= hbh^{-1}\, ,\\
c' &=hch^{-1}\, .
\end{align}
The equivalence classes of colorings are the conjugacy classes of the triple $(a,b,c)$. Notice that if the group $G$ is abelian, there are exactly $G^3$ equivalence classes of colorings.

\medskip

{\bf Computation for the cyclic group $\Zz_N$}

\medskip

 Let us recall that the cohomology group $H^3(\Zz_N,U(1))$ is $\Zz_N$ and is generated by the cohomology class of $\alpha_N$ defined as follows (See \cite{Moore})~:

\begin{equation}
\alpha_N(x,y,z)=exp((2i\pi/N^2)\ov{z}(\ov{x}+\ov{y}-\ov{xy}))\, , \label{cocycle}
\end{equation}
where $\ov{x}$ denotes the integer between $0$ and $N-1$ representing an element $x\in\Zz_N$. We have :

$$
\alpha_N(x,y,z)=\left\{
\begin{array}{cc}
1 & \mbox{when $x$ or $y$ or $z$ is equal to 1}\, ,\\
1 & \mbox{when $\ov{x}+\ov{y}<N$}\, ,\\
exp(2\pi\ov{z}/N) & \mbox{when $\ov{x}+\ov{y}\geq N$}
\end{array}
\right.
$$

In the case of Group categories the Turaev-Viro invariant and the Dijkgraaf Witten invariant are the same \cite{TVDW}, it follows ~:
\begin{equation}
  TV_{\Zz_N,\alpha_N}(\torust)=\frac{1}{N}\sum_{(a,b,c)\in \Zz_N^3}\frac{\alpha_N(a,b,c)\alpha_N(b,c,a)\alpha_N(c,a,b)}{\alpha_N(a,c,b)\alpha_N(c,b,a)\alpha_N(b,a,c)}=N^2\, .
\end{equation}
The homotopical Turaev-Viro invariant is~:
\begin{equation}
HTV_{\Zz_N,\alpha_N}(\torust,\_)=\left(\frac{1}{N}\frac{\alpha_N(a,b,c)\alpha_N(b,c,a)\alpha_N(c,a,b)}{\alpha_N(a,c,b)\alpha_N(c,b,a)\alpha_N(b,a,c)} \right)_{(a,b,c)\in \Zz_N^3}=\left(\frac{1}{N}\right)_{(a,b,c)\in \Zz_N^3}
\end{equation}

\subsubsection{The quantum group $U_q(\frak{sl}_2)$}

\subsubsection*{{\bf Colorings}}

Let $r\geq 3$, $A$ be a $2r$-th root of unity such that $A^2=q$ is a $r$-th root of unity. Let $\tilde{c}$ be a coloring of the triangulation of $\torust$, the coloring $\tilde{c}$ is determined by the 7-tuple $(a,b,c,d,e,f,g)$ which verifies~:
\begin{itemize}
\item[(i)] $0\leq a,b,c,d,e,f,g,h \leq r-2$\,
\item[(ii)] $(a,b,f)$, $(b,c,d)$, $(a,c,e)$, $(c,f,g)$ are admissible triple, \\
\end{itemize}

\subsubsection*{{\bf Gauge actions}}

Let us describe the equivalence classes of colorings of the 3-torus.  Let $\tilde{c}=(a,b,c,d,e,f,g)$ be a coloring of the 3-torus, we denote by $\cs{\tilde{c}}=(\cs{a},\cs{b},\cs{c},\cs{d},\cs{e},\cs{f},\cs{g})$ the coloring obtained by projection to the graduator $\Zz_2$. The coloring $\cs{\tilde{c}}$ is a $\Zz_2$-coloring. Using the example for group categories, we know that $\cs{\tilde{c}}$ is determined by the triple $(\cs{a},\cs{b},\cs{c})$. Furthermore since $\Zz_2$ is an abelian group, two colorings $\tilde{c}=(a,b,c,d,e,f,g)$ and $\tilde{c}'=(a',b',c',d',e',f')$ are equivalent for the gauge action if and only if $(\cs{a},\cs{b},\cs{c})=(\cs{a'},\cs{b'},\cs{c'})$. Thus the set of equivalence classes of a coloring $\tilde{c}=(a,b,c,d,e,f,g)$ is given by the class of $(\cs{a},\cs{b},\cs{c})$. It follows~:

\begin{align}
HTV_{U_q(\frak{sl}(2))}(\torust,\_)&=\left(\frac{-2r}{(A-A^-1)^2}\sum_{\substack{\tilde{c}=(a,b,c,d,e,f,g,h)\\\cs{a}=i,\cs{b}=j\cs{c}=k}}w_{\tilde{c}} W_{\tilde{c}}\right)_{(i,j,k)\in \Zz_2^3}\, ,\label{HTV:torus}
\end{align}
with $w_{\tilde{c}}=\dim(a)\dim(b)...\dim(h)$ and $W_{\tilde{c}}= \acl a & b & f \\ c& g& d  \acr^2\acl a & c & e \\ b& g& d  \acr^2\acl a & b & f \\ g& c& e  \acr^2$. The homotopy classes are given by the equivalence classes of colorings~:
$$(1,1,1), (1,1,-1), (1,-1,1), (-1,1,1), (-1,1,-1), (-1,-1,1), (1,-1,-1), (-1,-1,-1)\in \Zz^3\, .$$

Let us prove that $HTV_{U_q(\frak{sl}(2))}(S^1\times S^1\times S^1,(-1,1,1))$ and $HTV_{U_q(\frak{sl}(2))}(S^1\times S^1\times S^1,(1,-1,1))$ are equal. The formula (\ref{HTV:torus}) gives~:
\begin{align*}
 HTV_{U_q(\frak{sl}(2))}(\torust,(-1,1,1))&= \sum_{\substack{a,b,c,d,e,f,g\\ \cs{a}=\cs{e}=\cs{f}=\cs{g}=-1\\\cs{b}=\cs{c}=\cs{d}=1}}w_{\tilde{c}}\acl a & b & f\\ c & g &d \acr^2\acl a & c & e\\ b & g &d \acr^2\acl a & b & f\\ g & c &e \acr^2\, ,\\
    HTV_{U_q(\frak{sl}(2))}(\torust,(1,-1,1))&= \sum_{\substack{a,b,c,d,e,f,g\\ \cs{b}=\cs{d}=\cs{f}=\cs{g}=-1\\\cs{a}=\cs{c}=\cs{e}=1}}w_{\tilde{c}}\acl a & b & f\\ c & g &d \acr^2\acl a & c & e\\ b & g &d \acr^2\acl a & b & f\\ g & c &e \acr^2\, .\\
\end{align*}
The 6j-symbols are invariant under the action of the alternated group $\frak{A}_4$ on the 0-simplexes, it gives the following relations (\cite{matveev}, \cite{TV})~:
\begin{equation*}
\acl i & j & k\\ l & m & n \acr\,  = \, \acl j & i & k\\ m & l & n \acr = \acl i & k & j\\ l & n & m \acr \, = \, \acl i& m & n\\ l & j & k \acr \, = \, \acl l & m & k\\ i & j & n \acr\, =\,  \acl l & j & n\\ i & m & k \acr\, ,
\end{equation*}
with $(i,j,k)$, $(i,m,n)$, $(j,l,n)$ and $(k,l,m))$ admissible triples. Using the above relations, one gets~:
\begin{align*}
    HTV_{U_q(\frak{sl}(2))}(\torust,(1,-1,1))&= \sum_{\substack{a,b,c,d,e,f,g\\ \cs{b}=\cs{d}=\cs{f}=\cs{g}=-1\\\cs{a}=\cs{c}=\cs{e}=1}}w_{\tilde{c}}\acl b & a & f\\ g & c &d \acr^2\acl b & c & d\\ a & g &e \acr^2\acl b & a & f\\ c & g & e \acr^2\\
    &=HTV_{U_q(\frak{sl}(2))}(\torust,(-1,1,1))\, .
\end{align*}

Similarly, we prove the following equalities~:
\begin{align*}
  HTV_{U_q(\frak{sl}(2))}(\torust,(1,1,-1))&=HTV_{U_q(\frak{sl}(2))}(\torust,(1,-1,1))\\
&=HTV_{U_q(\frak{sl}(2))}(\torust,(-1,1,1)) \\
\end{align*}
and
\begin{align*}
HTV_{U_q(\frak{sl}(2))}(\torust,(1,-1,-1))&=HTV_{U_q(\frak{sl}(2))}(\torust,(-1,-1,1))\\
&=HTV_{U_q(\frak{sl}(2))}(\torust,(-1,1,-1))\, .
\end{align*}
\subsubsection*{{\bf Table}}

\begin{center}
\begin{tabular}{|c|c|c|}
\hline
r & $TV_{U_q(\frak{sl}(2))}$ & $HTV_{U_q(\frak{sl}(2))}$ \\
\hline
3 & $4$ & $(1/2,1/2,1/2,1/2,1/2,1/2,1/2,1/2)$ \\
\hline
4 & 9 & $(2,1,1,1,1,1,1,1)$ \\
\hline
5 & 16 & $(2,2,2,2,2,2,2,2)$\\
\hline
6 & 25 & $(4,3,3,3,3,3,3,3)$ \\
\hline
\end{tabular}
\end{center}

\subsection{The lens spaces}

Lens spaces $L(p,q)$, with 0<q<p and (p,q)=1, are oriented compact 3-manifolds, which result from identifying on the sphere $S^3=\{(x,y)\in \Cc^2\mid \cs{x}^2+\cs{y}^2=1\}$  the points which belong to the same orbit under the action of the cyclic group $\Zz_p$ defined by $(x,y)\ap (wx,w^qy)$ with $w=exp(2i\pi/p)$.

A singular triangulation of $L(p,q)$ is obtained by gluing together $p$ tetrahedra $(a_i,b_i,c_i,d_i)$, $i=0,...,p-1$ according to the following identification of faces $(i+1$ and $i+q$ are understood modulo p)~:
\begin{eqnarray}
(a_i,b_i,c_i)&=&(a_{i+1},b_{i+1},c_{i+1}) \label{i+1}\\
(a_i,b_i,c_i)&=&(b_{i+q},c_{i+q},d_{i+q}) \label{i+q}
\end{eqnarray}
The identification of (\ref{i+1}) can be realized by embedding the $p$ tetrahedra in Euclidean three-space, leading to a prismatic solid with $p+2$ 0-simplexes $a,b,c_i$, $2p$ external faces, $3p$ external edges and one internal axis $(a,b)$. Then formula (\ref{i+q}) is interpreted as the identification of the surface triangles $(a,c_i,c_{i+1})$ and $(b,c_{i+q},c_{i+1+q})$.

\subsubsection{Group categories}

\subsubsection*{{\bf Colorings}}

Let $G$ be a finite group, $\alpha\in H^3(G,\Bbbk^*)$ and $c$ be a coloring of the singular triangulation of $L(p,q)$ described above. We set $g=c(ab)$, $h_i=c(bc_i)$ and $k_i=c(c_ic_{i+1})$. By definition of the colorings, we have the following relations~:
\begin{align}
c(ac_i)&=c(ab)c(bc_i)=gh_i\label{rel:col1}\\
h_{i+1}&=c(bc_{i+1})=c(bc_i)c(c_ic_{i+1})=h_ik_i \label{rel:col2}
\end{align}

 The identification of the 2-simplexes $(a,c_i,c_{i+1})$ and $(b,c_{i+q},c_{i+1+q})$ gives~:
\begin{eqnarray}
k_i&=k_{i+q} \label{rel:k}\\
c(ac_i)&=h_{i+q}\label{rel:h}
\end{eqnarray}

The relations (\ref{rel:h}) and (\ref{rel:col1}) give : $h_i=h_{i+pq}=c(ac_{i+(p-1)q})=...=g^ph_i$. It implies that $g^p=e$. Since $(p,q)=1$, the relation (\ref{rel:k}) implies that $k_i$ is independent of $i$, we set $k_i=k\in G$. There exists an integer $n$ such that $n$ is the inverse of $q$ modulo $p$. One gets : $g^nh_i=h_{i+nq}=h_{i+1}$. By induction, it follows : $h_i=g^{in}h_0$. We set $h_0=h\in G$. The relation $g^nh_i=h_{i+1}$ compared to (\ref{rel:col2}) gives $k=h_i^{-1}g^nh_i=h^{-1}g^nh$. Conversely, the data $g,h\in G$ with $g^p=e$ determines a coloring of the singular triangulation of $L(p,q)$ through the formulas~:
\begin{align}
  g_i&=g\, , \label{collens:g}\\
  h_i &= g^{in}h\, ,  \label{collens:h} \\
  k_i &= h^{-1}g^{n}h\, . \label{collens:k}
\end{align}

\medskip

{\bf Gauge actions}

\medskip

The data $(g,h)\in G$ with $g^p=e$ defines a coloring of the above singular triangulation of $L(p,q)$. All the gauge actions on the colorings are on the form~:

\begin{eqnarray}
\delta.g&= \delta(a).g.\delta(a)^{-1}\, ,\\
\delta.h_i&= \delta(a).h_i.\delta(c_i)\, ,
\end{eqnarray}

with $\delta~:~T^0 \fd G$. Two colorings $(g,h)$ and $(g',h')$ are equivalent under the gauge actions if and only if there exist $x,y\in G$ such that~:~$g'=xgx^{-1}$ and $h'=xhy^{-1}$. We can notice that if $G$ is an abelian group then the equivalence class of the coloring $(g,h)$ is the set $\{(g,h)\}_{h\in G}$.

\medskip

{\bf Computation for $\Zz_N$.}

\medskip

In this case, the Turaev-Viro invariant is the Dijkgraaf-Witten invariant \cite{TVDW}~:

\begin{equation}
TV_{\Zz_N,\alpha_N}(L(p,q))=\frac{1}{N^2}\sum_{g,h\in \Zz_N, g^p=e}\prod_{i=0}^{p-1}\alpha_N(g,g^{in}h,h^{-1}g^nh)\, ,
\end{equation}
with $\alpha_N$ (\ref{cocycle}) the 3-cocycle which generates $H^3(\Zz_N,U(1))$.

\medskip

{\bf First case $p\nmid N$}

\medskip

In this case the set of colorings is $\{(1,h)\}_{h\in \Zz_N}$ and there is only one equivalent class of coloring. It follows that : $TV_{\Zz_N,\alpha_N}(L(p,q))=HTV_{\Zz_N,\alpha_N}(L(p,q),0)$, with $0\in [L(p,q),B\Zz_N]$ the trivial homotopy class. We obtain~:
$$
TV_{\Zz_N,\alpha_N}(L(p,q))=\frac{1}{N^2}\sum_{h\in \Zz_N}\prod_{i=0}^{p-1}\alpha_N(1,h,1)=\frac{1}{N}\, .
$$

\medskip

{\bf Second case $p\mid N$}

\medskip

In this case the number of homotopy classes is $\#\{g\in \Zz_N\mid g^p=1\}$. The invariant $HTV_{\Zz_N,\alpha_N}$ is :
\begin{equation}
HTV_{\Zz_N,\alpha_N}(L(p,q),-)=(\frac{1}{N},\frac{1}{N^2}\sum_{h\in  \Zz_N}\prod_{i=0}^{p-1}\alpha(g_1,g_1^{in}h,g_1^n),...,\frac{1}{N^2}\sum_{h\in  \Zz_N}\prod_{i=0}^{p-1}\alpha(g_k,g_k^{in}h,g_k^n))\, ,
\end{equation}
with $g_1,...,g_k\in \Zz_N$ such that $g_i^p=1$ for all $1\leq i \leq k$. We refer to the Section \ref{sec:table} for some values.

\subsubsection{$U_q(\mathfrak{sl}_2)$ with $q$ root of unity}

\subsubsection*{{\bf Colorings}}

Let us recall that for the previously singular triangulation of $L(p,q)$ we denote by $(a,b,c_i,c_{i+1})$, with $0\leq i\leq p-1$, the 3-simplexes. Let $c$ be a coloring of this singular triangulation, we set~:
\begin{align*}
  c(ab)&= X\, ,\\
  c(bc_i)&=Y_i\, ,\\
  c(c_ic_{i+1})&= Z_i\, ,\\
  c(ac_i)&=K_i\, . \\
\end{align*}
The identification of the 2-simplexes $(a,c_i,c_{i+1})$ and $(b,c_{i+q},c_{i+q+1})$ gives the following relations~:
\begin{align*}
K_i&= Y_{i+q}\, ,\\
Z_i&= Z_{i+q}\, ,
\end{align*}
for all $0\leq i \leq p-1$. Since $(p,q)=1$, there exists an integer $n$ such that : $nq=1$ mod $p$. It follows : $Z_{i+1}=Z_{i+nq}=Z_i$. We set $Z_i=Z_0=Z$. Thus a coloring of $L(p,q)$ is determined by the colors of the edges : $(ab)$, $(c_ic_{i+1})$ and $(bc_i)$. From now on, a coloring $c$ of $L(p,q)$ will be denoted by $(c(ab),c(c_ic_{i+1}),c(bc_i))$.

\subsubsection{{\bf Gauge actions}}

For every scalar object $X$, $|X|$ denotes the image of the scalar object $X$ by the projection map $|?|:\Lambda_{U_q(\mathfrak{sl}_2)}\fd \Zz_2$.

Let $n$ be an integer such that $nq=1$ mod $p$. For every coloring $c=(X,Z,Y_i)$, $(|X|,|Z|,|Y_i|)$ is a $\Zz_2$-coloring  of $L(p,q)$. From the computation int he case of group categories, one gets~:
\begin{align*}
|Y_i|&=|X|^{in}h\, ,\\
|Z|&=|X|^n\, ,
\end{align*}
with $h\in \Zz_2$. Using the case of group categories, we know that two colorings $(|X|,|Z|,|Y_i|)$ and $(|X'|,|Z'|,|Y_i'|)$ are equivalent if and only if $|X|=|X'|$. It results that two colorings $c=(X,Z,Y_i)$ and $c'=(X',Z',Y_i')$ are equivalent if and only if $|X|=|X'|$. Thus the parity of $X$ describes the equivalent classes. There exists at most two homotopy classes, the trivial homotopy class $0\in[L(p,q),B\Zz_2]$ corresponds to the equivalence class of the coloring $(X,Z,Y_i)$ with $\cs{X}=1$. Then the Turaev-Viro invariant can be written in the following way~:
\begin{align}
TV_{U_q(\mathfrak{sl}_2)}(L(p,q))&=\Delta_{U_q(\mathfrak{sl}_2)}^{-2}\sum_{c=(X,Z,Y_i)}w_cW_c \nonumber\\
&= \Delta_{U_q(\mathfrak{sl}_2)}^{-2}\left(\sum_{\substack{c=(X,Z,Y_i)\\ \cs{X}=1}}w_cW_c+\sum_{\substack{c=(X,Z,Y_i)\\\cs{X}=-1}}w_cW_c\right)\, .
\end{align}
We denote by $HTV_0(L(p,q))$ (resp. $HTV_1(L(p,q))$) the state sum $\dis{\Delta_{U_q(\mathfrak{sl}_2)}^{-2}\sum_{\substack{c=(X,Z,Y_i)\\ \cs{X}=1}}w_cW_c}$ (resp. $\dis{\Delta_{U_q(\mathfrak{sl}_2)}^{-2}\sum_{\substack{c=(X,Z,Y_i)\\ \cs{X}=-1,\cs{X}^p=1}}w_cW_c}$). The state sum $HTV_0$ is the homotopical Turaev-Viro invariant for the trivial homotopy class, and $HTV_1$ is the homotopical Turaev-Viro obtained for the other homotopy class.

\medskip

{\bf If $p$ is odd.}

\medskip

In this case the set of colorings is $(X,Z,Y_i)_{\cs{X}=1,Z,Y_i}$, it follows that there is only one homotopy class given by the equivalence class of the coloring $(0,Z,Y_i)$. We obtain : $TV(L(p,q))=HTV_0(L(p,q))$ and $HTV_1(L(p,q)=0$.

\subsubsection{The case $r=3$}

The set of irreducible scalar objects (up to isomorphisms) consists of two elements $0$ and $1$. Up to permutation there are only two admissible (unordered) triples : $(0,0,0)$ and $(0,1,1)$.

Let $A$ be a 6th root of unity with $A^2=q$ a 3rd root of unity. It follows that~: $q^2+q+1=0$. Set $\epsilon=A+A^{-1}\not = 0$. We obtain :
\begin{align*}
\epsilon^2 &= q+2+q^{-1}\\
&= q+1-q\\
&=1
\end{align*}
Thus $\epsilon =1$ if the real part of $A$ is positive and $\epsilon=-1$ if the real part of $A$ is negative. We have~:%
\begin{align*}
  \dim(0)&=1\, ,\\
  \dim(1)&=-\epsilon\, , \\
  \Delta_{U_q(\mathfrak{sl}_2)}&= 2\, .
\end{align*}
Each admissible 6-tuple may be transformed by the action of the alternating $\mathfrak{A}_4\subset \mathfrak{S}_4$ into one of the three 6-tuples : $(0,0,0,0,0,0)$, $(1,1,0,1,1,0)$ and $(0,1,1,1,0,0)$. We obtain the following 6j-symbols~:
\begin{align*}
  \left|\begin{array}{ccc}
  0 & 0 & 0\\
  0 & 0 & 0
  \end{array}\right| &= 0\, ,\\
  \left|\begin{array}{ccc}
  1 & 1 & 0\\
  1 & 1 & 0
  \end{array}\right| &= -\epsilon\,   ,\\
  \left|\begin{array}{ccc}
  0 & 1 & 1\\
  1 & 0 & 0
  \end{array}\right| &=\left\{\begin{array}{cc}\imath & \mbox{if $\epsilon = 1$}\, ,\\
  1 & \mbox{if $\epsilon=-1$}\, .\end{array}\right. \\
\end{align*}

The homotopical Turaev-Viro invariant is~:

\begin{equation*}
  (HTV_0(L(p,q)),  HTV_1(L(p,q)))=\left\{\begin{array}{cc}(1/2,0) & \mbox{if $p$ is odd}\\ (1/2,(-\epsilon)^{p/2}/2) & \mbox{if $p$ is even}\end{array}\right.
\end{equation*}

\section{Splitting of the Turaev-Viro TQFT}\label{sec:splitting}

In this section, we will build a splitting of the Turaev-Viro TQFT. To obtain this splitting, we will use the homotopical Turaev-Viro invariant (Theorem \ref{thm:inv}). Throughout this section, the category $\C$ is a spherical category with an invertible dimension in $\Bbbk$.

\subsection{The Turaev-Viro TQFT}

Let $\Sigma$ and $\Sigma'$ be two oriented closed surfaces, a \emph{cobordism from $\Sigma$ to $\Sigma'$} is a 3-manifold whose boundary is the disjoint union : $\overline{\Sigma}\coprod \Sigma'$. Let $M$ and $M'$ be two cobordisms from $\Sigma$ to $\Sigma'$, $M$ and $M'$ are equivalents if there exists an isomorphism between $M$ and $M'$ such that it preserves the orientation and its restriction to the boundary is the identity.

The \emph{cobordism category} is the category where objects are closed and oriented surfaces and morphisms are equivalent classes of cobordisms. The cobordism category is denoted by $Cob_{1+2}$. The disjoint union and the empty manifold $\emptyset$ define a strict monoidal structure on $Cob_{1+2}$.

A \emph{TQFT} is a monoidal functor from the cobordism category  to the category of finite dimensional vector spaces.

We will recall the construction  of the Turaev-Viro TQFT. Let $\Sigma$ be an oriented closed surface and $T$ be a triangulation of $\Sigma$. We associate to the pair $(\Sigma,T)$ a vector space $\dis{V_{\C}(\Sigma,T)=\bigoplus_{c\in Col(T)}\bigotimes_{f\in T_0^2}V(f,c)}$, where $V(f,c)=\ho{\Bbbi}{c(01)\pt c(12)\pt c(20)}$ for every $f=(012)$. The vector space $V(f,c)$ does not depend on the choice of a numbering which respects the orientation. The vector space $\dis{\bigotimes_{f\in T_0^2}V(f,c)}$ is dual to $\dis{\bigotimes_{f\in T_0^2}V(\ov{f},c)}$, the duality pairing $\dis{\Omega_{\C,c} : \bigotimes_{f\in T_0^2}V(f,c)\pt_{\Bbbk} \bigotimes_{f\in T_0^2}V(\ov{f},c)}$ is induced by the non degenerate bilinear form (\ref{bilinear})~:
\begin{align*}
\omega_{\C} : Hom_{\C}(\Bbbi,X\pt Y \pt Z)\pt_{\Bbbk} Hom_{\C}(\Bbbi,\xd{Z}\pt \xd{Y} \pt \xd{Z})&\fd \Bbbk^* \\
f\pt g & \ap tr(\xd{f}g)\, ,\nonumber
\end{align*}
for every objects $X$, $Y$ and $Z$. The duality pairings corresponding to all colorings of $T$ determine a bilinear form $\Omega_{\C} : V_{\C}(\Sigma,T)\pt_{\Bbbk} V_{\C}(\ov{\Sigma},T)\fd \Bbbk$ by the formula~:
\begin{equation}\label{bilenargen}
\Omega_{\C}(\bigoplus_{c\in Col(T)}x_c,\bigoplus_{c\in Col(T)}y_c)=\sum_{c\in Col(T)}\omega_{\C,c}(x_c,y_c)\, ,
\end{equation}
where $x_c\in \prod_{f\in T_0^2}V(f,c)$ and $y_c\in \prod_{f\in T_0^2}V(\ov{f},c)$ for all coloring $c\in Col(T)$. The bilinear form (\ref{bilenargen}) is non degenerate and symmetric.

Let $\Sigma$ (resp. $\Sigma'$) be an oriented surface endowed with a triangulation $T$ (resp. $T'$) and $M$ be a cobordism from $\Sigma$ to $\Sigma'$, for every colorings $c\in Col(T)$ and $c'\in Col(T')$ we have the following vector : $TV_{\C}(M,c,c')\in V_{\C}(\ov{\Sigma},T,c)\pt V_{\C}(\Sigma',T',c')\cong V_{\C}(\Sigma,T,c)^*\pt V_{\C}(\Sigma',T',c')$. The vector spaces $V_{\C}(\Sigma,T,c)$ and $V_{\C}(\Sigma',T',c')$ are finite dimensional vector spaces, thus we can build the following linear map :
$$\ov{TV_{\C}}(M)_{c,c'} : V_{\C}(\Sigma,T,c)\fd V_{\C}(\Sigma',T',c')\, ,$$
thus the matrix $\bigg( \ov{TV_{\C}}(M)_{c,c'}\bigg)_{c\in Col(T),c'\in Col(T')}$ defines a  linear map :
$$[M]= \bigg( \ov{TV_{\C}}(M)_{c,c'}\bigg)_{c\in Col(T),c'\in Col(T')}: V_{\C}(\Sigma,T)\fd V_{\C}(\Sigma',T')\, .$$
By construction of the Turaev-Viro invariant (Theorem 1.8 \cite{Tu}), we have the following relation : let $\Sigma$, $\Sigma'$ and $\Sigma''$ be closed surfaces endowed with the triangulations $T$, $T'$ and $T''$,  for every cobordisms $M : (\Sigma,T) \fd (\Sigma',T')$ and $M' : (\Sigma',T') \fd (\Sigma'',T'')$ we have : $TV_{\C}(M'\cup_{\Sigma'}M)=\textrm{contr}_{\Sigma}(TV_{\C}(M')\pt TV_{\C}(M))$, where $\textrm{contr}_{\Sigma}$ is the contraction~:
\begin{equation}\label{contractionTQFT}
V_{\C}(\ov{\Sigma},T)\pt_{\Bbbk}V_{\C}(\Sigma',T')\pt_{\Bbbk}V_{\C}(\ov{\Sigma'},T')\pt_{\Bbbk}V_{\C}(\Sigma'',T'')\fd V_{\C}(\ov{\Sigma},T)\pt_{\Bbbk}V_{\C}(\Sigma'',T'')
\end{equation}
induced by the form $\Omega_{\C}$ in $V_{\C}(\Sigma',T')$. It follows that~: $[M'\cup_{\Sigma'}M]=[M']\circ[M]$. Furthermore the map $[\Sigma\times I] : V_{\C}(\Sigma,T) \fd V_{\C}(\Sigma,T)$ is an idempotent denoted by $p_{\Sigma,T}$. Set $\m{V}_{\C}(\Sigma,T)=\im(p_{\Sigma,T})$ and for every cobordism $M : \Sigma\fd \Sigma'$ we denote by $\m{V}_{\C}(M)=[M]_{\im(p_{\Sigma,T})}$ the restriction of $[M]$ to $\im(p_{\Sigma,T})$. The vector space $\m{V}_{\C}(\Sigma,T)$ is independent on the choice of the triangulation $T$. Indeed for every triangulations $T$ and $T'$ of $\Sigma$, the equivalence class of the cobordism $\Sigma\times I$, where the surface $\Sigma\times \{0\}$ is endowed with the triangulation $T$ and the surface $\Sigma\times \{1\}$ is endowed with the triangulation $T'$, is an isomorphism; the inverse  is the cobordism $\Sigma\times I$ where $\Sigma\times \{0\}$ (resp. $\Sigma\times \{1\}$) is endowed with the triangulation $T'$ (resp. $T$). Thus the linear map defined by this cobordism is an isomorphism between $\m{V}_{\C}(\Sigma,T)$ and $\m{V}_{\C}(\Sigma,T')$. From now on we will denote by $\m{V}_{\C}$ the Turaev-Viro TQFT, for every closed surface $\Sigma$ we denote by $\m{V}_{\C}(\Sigma)$ the vector space associated to $\Sigma$ and for every cobordism $M$ we denote by $\m{V}_{\C}(M)$ the linear map associated to $M$.

\subsection{The splitting of the Turaev-Viro TQFT}

From now on, for every homotopy classes $x\in [\Sigma,B\grad]$ and $x'\in [\Sigma',B\grad]$ we denote by $[M,B\grad]_{(\Sigma,x),(\Sigma',x')}$ the set of homotopy classes of $[M,B\grad]$ such that the homotopy class of the restriction to $\Sigma$ (resp. $\Sigma'$) is $x$ (resp. $x'$).

For every oriented surface $\Sigma$ endowed with a triangulation $T$, we can decompose the vector space $V_{\C}(\Sigma,T)$ in the following way :
$$V_{\C}(\Sigma,T)=\bigoplus_{x\in [\Sigma,B\grad]}\bigoplus_{c\in Col_x(T)}V_{\C}(\Sigma,T,c)=\bigoplus_{x\in [\Sigma,B\grad]}V_{\C}(\Sigma,T,x) \, ,$$
where $V_{\C}(\Sigma,T,x)$ is the vector space $\dis{\bigoplus_{c\in Col_x(T)}V_{\C}(\Sigma,T,c)}$.

Let $M$ be a cobordism from $(\Sigma,T)$ to $(\Sigma',T')$, $c$ be a coloring of $T$ and $c'$ be a coloring of $T'$. For every homotopy class $y\in [M,B\grad]_{(\Sigma,x_c),(\Sigma',x_{c'})}$, we know that $HTV_{\C}(M,y,c,c')$ is a vector in $V_{\C}(\Sigma,T,c)^*\pt V_{\C}(\Sigma',T',c')$. This vector defines a linear map :
$$
\ov{HTV_{\C}}(M,y,c,c') : V_{\C}(\Sigma,T,c) \fd V_{\C}(\Sigma,T',c')\, .
$$

Let $x\in [\Sigma,B\grad]$ and $x'\in [\Sigma,B\grad]$, for every $y\in [M,B\grad]_{(\Sigma,x),(\Sigma',x')}$ the matrix $\bigg(\ov{HTV_{\C}}(M,y,c,c')\bigg)_{c\in Col_x(T), c'\in Col_x(T')}$ defines a linear map :

$$
\bigg(\ov{HTV_{\C}}(M,y,c,c')\bigg)_{c\in Col_x(T), c'\in Col_x(T')} : V_{\C}(\Sigma,T,x) \fd V_{\C}(\Sigma',T',x')\, .
$$
This map is denoted by $\widetilde{HTV_{\C}}(M,y)_{x,x'}$.

\subsection*{Compositions}

Let $M$ be a cobordism from $(\Sigma,T)$ to $(\Sigma',T')$, $M'$ be a cobordism from $(\Sigma',T')$ to $(\Sigma'',T'')$, $x\in [\Sigma,B\grad]$ and $x''\in [\Sigma'',B\grad]$. By construction of the homotopical Turaev-Viro invariant for every $y\in [M,B\grad]_{(\Sigma,x),(\Sigma',x')}$, $y'\in [M',B\grad]_{(\Sigma',x'),(\Sigma'',x'')}$, $c\in Col_x(T)$ and $c''\in Col_{x"}(T")$, we have :
$$HTV_{\C}(M'\cup_{\Sigma'} M,y\cup y',c,c'')=\textrm{contr}_{\Sigma',x'}(HTV_{\C}(M,y,c,c')\pt HTV_{\C}(M',y',c',c''))\, ,$$
where $\textrm{contr}_{\Sigma',x'}$ is the contraction~:
$$
V_{\C}(\ov{\Sigma},T,x)\pt_{\Bbbk}V_{\C}(\Sigma',T',x')\pt_{\Bbbk}V_{\C}(\ov{\Sigma'},T',x')\pt_{\Bbbk}V_{\C}(\Sigma'',T'',x'')\fd V_{\C}(\ov{\Sigma},T,x)\pt_{\Bbbk}V_{\C}(\Sigma'',T'',x'')
$$
induced by the form $\Omega_{\C}$ in $V_{\C}(\Sigma',T',x')$ and with $y\cup y'(x) = \left\{\begin{array}{cc}y(x) & \mbox{if $x\in M$}\, ,\\ y'(x) & \mbox{if $x\in M'$}\, . \end{array}\right.$. It follows~:
$$
\sum_{c'\in Col_{x'}(T')}\ov{HTV_{\C}}(M',c',c'',y')\circ \ov{HTV_{\C}}(M,c,c',y)l=\ov{HTV_{\C}}(M'\cup_{\Sigma}M,c,c''y\cup y')\, .
$$

The composition is well defined, let us show that the morphism $(\Sigma\times I)_{T,T} : (\Sigma,T) \fd (\Sigma,T)$ defines an idempotent of $V_{\C}(\Sigma,T,x)$.

\subsection*{Idempotents}

Let $\Sigma$ be a surface, the inclusion $\Sigma \hookrightarrow \Sigma\times I$ is a deformation retract, thus there exists a unique homotopy class $y\in [\Sigma\times I,B\grad]$ such that the homotopy class of the restriction to $\Sigma\times \{0\}$ is $x$. More precisely, $y$ is the homotopy class of the following map :
\begin{align*}
\Sigma \times I &\fd B\grad \\
(z,t) & \ap x(z).
\end{align*}

We denote by $1_x$ this homotopy class. Assume that there exists an homotopy class $y\in[\Sigma\times I,B\grad]_{(\Sigma,x),(\Sigma,x')}$ then there exists a map : $Y : \Sigma\times I \fd B\grad$ such that $Y_{\Sigma\times \{0\}}$ (resp. $y_{\Sigma\times \{1\}}$) is homotopic to $x$ (resp. $x'$). It follows that the linear map $\widetilde{HTV_{\C}}(M\times I,y)_{x,x'}$ is defined if and only if the homotopy classes $x$ and $x'$ are the same. When this linear map is defined then $1_x$ is the unique homotopy class of $[\Sigma\times I]_{(\Sigma,x),(\Sigma',x')}$.  We denoted by $p_{\Sigma,T,x}$ the linear map $\widetilde{HTV_{\C}}(\Sigma\times I,1_x)_{x,x}$. By definition of the composition, this endomorphism is an idempotent.

\begin{lem}\label{lem:idemspli}
For every surface $\Sigma$ endowed with a triangulation $T$, we have :
\begin{equation*}
p_{\Sigma,T}=\bigoplus_{x\in [\Sigma,B\grad]}p_{\Sigma,T,x}.
\end{equation*}
\end{lem}

\pf
For every 3-manifold $M$ with boundary $\Sigma$, for every triangulation $T$ of $\Sigma$ and for every coloring $c\in Col(T)$, we have : $\dis{TV_{\C}(M,c)=\sum_{x\in [M,B\grad]_{\Sigma,x_c}}HTV_{\C}(M,x,c)}$. Thus if $M=\Sigma\times I$, then $\dis{TV_{\C}(\Sigma\times I,c,c')=\sum_{y\in[\Sigma\times I,B\grad]_{(\Sigma,x_c),(\Sigma,x_{c'})}}HTV_{\C}(\Sigma\times I,y,c,c')}$. Previously we have shown that if the homotopy classes $x_c$ and $x_{c'}$ are different then $[\Sigma\times I,B\grad]_{(\Sigma,x_c),(\Sigma,x_{c'})}$ is the empty set and if the homotopy classes $x_c$ and $x_{c'}$ are the same then $[\Sigma\times I,B\grad]_{(\Sigma,x_c),(\Sigma,x_{c'})}=\{1_{x_c}\}$. Thus if $c,c'\in Col_x(T)$ then $TV_{\C}(\Sigma\times I,c,c')=HTV_{\C}(\Sigma\times I,1_{x},c,c')$. Furthermore if $c\in Col_x(T)$ and $c'\in Col_{x'}(T)$ with $x\not=x'$ then $TV_{\C}(\Sigma\times I,c,c')=0$. It follows that $\dis{p_{\Sigma,T}=\bigoplus_{x\in [\Sigma, B\grad]}p_{\Sigma,T,x}}$.

\qed

For every closed surface $\Sigma$ endowed with a triangulation $T$, we set : $\m{W}_{\C}(\Sigma,T,x)=\im(p_{\Sigma,T,x})\,.$ Let $M$ be a cobordism from $(\Sigma,T)$ to $(\Sigma',T')$, for every $x\in [\Sigma,B\grad]$, $x'\in [\Sigma',B\grad]$ and $y\in [M,B\grad]_{(\Sigma,x),(\Sigma',x')}$, we denote by $\m{W}_{\C}(M,y)_{x,x'}$ the restriction of $\widetilde{HTV_{\C}}(M,y)_{x,x'}$ to the vector spaces $\m{W}_{\C}(\Sigma,T,x)$ and $\m{W}_{\C}(\Sigma',T',x')$. By definition of the composition, $\m{W}_{\C}(M,y)_{x,x'}$ is a linear map from $\m{W}_{\C}(\Sigma,T,x)$ to $\m{W}_{\C}(\Sigma',T',x')$.

\medskip

Let us show that $\m{W}_{\C}(\Sigma,T,x)$ doesn't depend on the choice of the triangulation. For every closed surface $\Sigma$ and for every triangulations $T$ and $T'$ of $\Sigma$, the linear map $\m{W}_{\C}(\Sigma\times I,1_x)_{x,x} : \m{W}_{\C}(\Sigma,T,x)\fd \m{W}_{\C}(\Sigma,T',x)$ is an isomorphism. Thus the space $\m{W}_{\C}(\Sigma,T,x)$ doesn't depend on the choice of the triangulation $T$, from now on we denote this vector space by $\m{W}_{\C}(\Sigma,x)$. Notice that if $T=T'$ then $\m{W}_{\C}(\Sigma\times I,1_x)_{x,x}=\id_{\m{W}_{\C}(\Sigma,T,x)}$.
\begin{theorem}\label{thm:decomp}
Let $\C$ be a spherical category. For every closed and oriented surface $\Sigma$, we have the following decomposition of the Turaev-Viro TQFT $\m{V}_{\C}$ : \begin{equation}\label{scindage:ev}
\m{V}_{\C}(\Sigma)=\bigoplus_{x\in [\Sigma,B\grad]}\m{W}_{\C}(\Sigma,x).
\end{equation}

For every cobordism $M : \Sigma_0 \fd \Sigma_1$ and for every $x_0\in [\Sigma_0,B\grad]$, $x_1\in [\Sigma_1,B\grad]$, we denote by $\m{V}_{\C}(M)_{x_0,x_1}$ the following restriction of the map $\m{V}_{\C}(M)$~:
$$
\xymatrix{
\m{V}_{\C}(\Sigma_0) \ar[r]^{\m{V}_{\C}(M)} & \m{V}_{\C}(\Sigma_1) \\
\m{V}_{\C}(\Sigma_0,x_0) \ar@{^{(}->}[u]  \ar[r]_{\m{V}_{\C}(M)_{x_0,x_1}} & \m{V}_{\C}(\Sigma_1,x_1) \ar@{^{(}->}[u] \, .
}
$$
We have the following splitting~:
\begin{equation}
\m{V}_{\C}(M)_{x_0,x_1}= \bigoplus_{y\in [M,B\grad]_{(\Sigma_0,x_0),(\Sigma_1,x_1)}}\m{W}_{\C}(M,y)_{x_0,x_1},\label{scindage:mo}
\end{equation}
\end{theorem}
\pf
The splitting (\ref{scindage:ev}) is a consequence of the lemma \ref{lem:idemspli}.

Let us show the decomposition (\ref{scindage:mo}). Let $M : (\Sigma_0,T_0)\fd (\Sigma_1,T_1)$, by construction $\m{V}_{\C}(M)$ is the restriction of the linear map $[M]$ to the image of the idempotent $p_{\Sigma_0,T_à}$ and the linear map is given by the matrix $(\ov{TV_{\C}}(M)_{c_0,c_1})_{c_0\in Col(T_0),c_1\in Col(T_1)}$ and for every colorings $c_0\in Col(T_0)$ and $c_1\in Col(T_1)$, we have : $$\dis{\ov{TV_{\C}}(M)_{c_0,c_1}=\sum_{y\in [M,B\grad]_{(\Sigma_0,x_{c_0}),(\Sigma_1,x_{c_1})}}\ov{HTV_{\C}}(M,y)_{c_0,c_1}}\, .$$ It follows that for every $x_0\in [\Sigma_0,B\grad]$ and $x_1\in [\Sigma_1,B\grad]$, the restriction of the map $[M]$ to the vector spaces : $V_{\C}(\Sigma_0,T_0,x_0)$ and  $V_{\C}(\Sigma_1,T_1,x_1)$ is equal to $\dis{\bigoplus_{y\in[M,B\grad]_{(\Sigma_0,x_0),(\Sigma_1,x_1)}}HTV_{\C}(M,y)_{x_0,x_1}}$. According to the lemma \ref{lem:idemspli}, the idempotent $p_{\Sigma_0,T_0}$  (resp. $p_{\Sigma_1,T_1}$) is splitting in the following way : $\dis{p_{\Sigma_0,T_0}=\bigoplus_{x\in[\Sigma,B\grad]}p_{\Sigma_0,T_0,x}}$ (resp. $\dis{p_{\Sigma_0,T_0}=\bigoplus_{x\in[\Sigma,B\grad]}p_{\Sigma_0,T_0,x}}$), thus we obtain the splitting (\ref{scindage:mo}).

\qed

A direct consequence of the theorem \ref{thm:decomp} is the following formula of the dimension of the vector space associated to a closed surface $\Sigma$~:
\begin{equation}\label{eq:dim}
\dim_{\Bbbk}(\m{V}_{\C}(\Sigma))=\sum_{x\in [\Sigma,B\grad]}\dim_{\Bbbk}(\m{W}_{\C}(\Sigma,x))\quad \textrm{and} \quad \dim_{\Bbbk}(\m{W}_{\C}(\Sigma,x))=HTV_{\C}(\Sigma\times S^1,1_x)\, .
\end{equation}

\subsection{The Turaev-Viro HQFT}

In this section we show that for every spherical category $\C$, the Turaev-Viro TQFT is obtained from a 2+1 dimensional HQFT whose target space is $B\grad$. Let us recall the definition of an HQFT.

\subsection*{$B$-manifolds}

Let $B$ be a $d$-dimensional manifold, a \emph{$d$-dimensional $B$-manifold} is a pair $(X,g)$ where $X$ is closed $d$-manifold and $g : X\fd B$ is a continuous map called \emph{characteristic map}.

A \emph{$B$-cobordism from $(X,g)$ to $(Y,h)$} is a pair $(W,F)$ where $W$ is a cobordism from $X$ to $Y$ and $F$ is a relative homotopy class of a map from $W$ to $B$ such that the restriction to $X$ (resp. $Y$) is $g$ (resp. $h$). From now on, we make no notational distinction between a (relative) homotopy class  and any of its representatives. Notice that if $B$ is a just a point $\{*\}$ then we recover the notion of cobordism.

We define the operation of gluing for $B$-cobordism. This notion is similar to the notion of gluing for cobordism. Let $(W,F) : (M,g)\fd (N,h)$ and $(W',F') : (N',h')\fd (P,k)$ be two $B$-cobordisms and $\Psi : N\fd N'$ be a diffeomorphism such that $h'\psi=h$. The composition of  $B$-cobordisms is defined in the following way~:~$(W',F')\circ(W,F)=(W'\cup W,F.F')$, where $F.F'$ is the following homotopy class :

$$
F.F'(x) = \left\{
\begin{array}{cc}
F(x) & x\in W \\
F'(x) & x\in W'\\
\end{array}
\right.
$$

Since $h'\Psi=h$, the map $F.F'$ is well defined.

The identity of $(X,g)$ is the $B$-cobordism $(X\times I,1_g)$, with $1_g$ the homotopy class of the map~:
\begin{align*}
X\times I &\fd B\\
(x,t) &\ap g(x)
\end{align*}

The disjoint union of $B$-cobordisms is defined in the same way of disjoint union of cobordisms.

The \emph{category of $d+1$ $B$-cobordisms} is the category whose objects are $d$-dimensional $B$-manifolds and morphisms are isomorphism classes of $B$-cobordisms. The category of $d+1$ $B$-cobordism is denoted by $Hcob(B,d+1)$, this is a strict monoidal category.

\subsection*{HQFTs}

A \emph{$d+1$ dimensional HQFT with target space $B$} is a monoidal functor from the category $Hcob(d+1,B)$ to the category of finite dimensional vector spaces.

Actually the vector space obtained from a $B$-manifold only depends on the manifold and the homotopy class of the characteristic map.
\begin{pro}
Let $F$ bet a $d+1$ dimensional HQFT with target space $B$ and $(X,g)$ be a $B$-manifold, then for every linear map $h : X\fd B$ homotopic to $g$ we have : $F(X,g)\cong F(X,h)$.
\end{pro}
\pf
Let $H$ be an homotopy between $g$ and $h$, then $(X\times I, H)$ is a $B$-cobordism from $(X,g)$ to $(X,h)$ and the pair $(X\times I,\tilde{H})$, where $\tilde{H}$ is the homotopy class of $H(x,1-t)$ is a $B$-cobordism from $(X,h)$ to $(X,g)$. Since $F$ is an HQFT, we have : $F(X\times I,H)\circ F(X\times I,\tilde{H})=F(X\times I, H\cup \tilde{H})$ with $H\cup \tilde{H}$ the homotopy class of the map :
\begin{align*}
X\times I &\fd B\\
(x,t) & \ap H\cup\tilde{H}=\left\{\begin{array}{cc}H(x,2t) & \mbox{ if
$0\leq t \leq \frac{1}{2}$}\\
H'(x,1-2t) & \mbox{ if $\frac{1}{2} \leq t \leq 1$}
\end{array}\right.
\end{align*}

Let us show that $H\cup \tilde{H}$ is homotopic to $1_g$. The map :
\begin{align*}
(\Sigma\times I) \times I & \fd \Sigma\\
((x,t),s)& \ap \left\{
\begin{array}{cc}
1_g(x,t) & \mbox{if $0\leq s$}
\leq
\frac{1}{2}, \\
H\cup H'(x,t(2s-1)) & \mbox{if $\frac{1}{2} \leq s \leq 1$}
\end{array}
\right.
\end{align*}
is an homotopy between $1_g$ and $H\cup \tilde{H}$. Thus the map $F(X\times I, H): F(X,g)\fd F(X,h)$ is an isomorphism.

\qed

\begin{theorem}\label{HQFT}
Let $\C$ be a spherical category. Set :
\begin{align}
\m{H}_{\C} : Hcob(B\grad,2+1) & \fd \textrm{vect}_{\Bbbk} \label{HQFT}\\
(\Sigma,g) & \ap \m{W}_{\C}(\Sigma,g),\nonumber\\
(M,F) & \ap  \m{W}_{\C}(M,F),\nonumber
\end{align}
where the vector space $\m{W}_{\C}(\Sigma,g)$ is defined for the homotopy class of $g$. The functor $\m{H}_{\C}$ is a $2+1$ dimensional HQFT with target space the classifying space $B\grad$.

The Turaev-Viro TQFT $\m{W}_{\C}$ is obtained from the HQFT $\m{H}_{\C}$  :

$$
\m{W}_{\C}(\Sigma)=\bigoplus_{x\in [\Sigma,B\grad]}\m{H}_{\C}(\Sigma,x).
$$
\end{theorem}
\pf
Let us show that $\m{H}_{\C}$ is a functor. Let $\Sigma$ (resp. $\Sigma'$ and $\Sigma''$) be a closed oriented surface endowed with a triangulation $T_{\Sigma}$ (resp. $T'$ and $T''$). For every cobordisms $(M,F) : (\Sigma,g)\fd (\Sigma',g')$ and $(M',F') : (\Sigma',g')\fd (\Sigma'',g'')$, we have shown that for every colorings $c\in Col_{x_g}(T_{\Sigma})$, $c'\in Col_{x_{g'}}(T_{\Sigma'})$ and $c''\in Col_{x_{g''}}(T_{\Sigma''})$, where $x_g$ (resp. $x_{g'}$, $x_{g''}$) is the homotopy class of $g$ (resp. $g'$, $g''$), and for every $y\in [M,B\grad]_{(\Sigma,x_g),(\Sigma',x_{g'})}$, $y'\in [M',B\grad]_{(\Sigma',x_{g'}),(\Sigma'',x_{g''})}$ we have :
$$\sum_{c'\in Col_{x_{g'}}(T')}\ov{HTV_{\C}}(M',c',c'',F')\circ \ov{HTV_{\C}}(M,c,c',F)=\ov{HTV_{\C}}(M'\cup_{\Sigma'}M,c,c'',F\cup F').$$
Thus
$$\widetilde{HTV_{\C}}(M',F)_{x_{g'},x_{g''}}\circ\widetilde{HTV_{\C}}(M,F)_{x_g,x_{g'}}=\widetilde{HTV_{\C}}(M'\cup_{\Sigma'}M,F'\cup F)_{x_g,x_{g''}}.$$
If we consider the restriction to the image of the idempotent $p_{\Sigma,T,x_g}$, we have : $\m{W}_{\C}(M',F')\m{W}_{\C}(M,F)=\m{W}_{\C}(M'\cup_{\Sigma}M,F'\cup F)$. Furthermore we have : $\m{W}_{\C}(\Sigma\times I, 1_x)=\id_{\m{W}_{\C}(\Sigma,x)}$.

Let us show that $\m{H}_{\C}$ is monoidal. Let $\Sigma$ (resp. $\Sigma'$) be a closed surface endowed with a triangulation $T$ (resp. $T'$), then the vector space :
$${V_{\C}(\Sigma\coprod \Sigma',T\cup T')=V_{\C}(\Sigma,T)\pt V_{\C}(\Sigma',T')}\, .$$
Thus for every homotopy classes $x\in [\Sigma,B\grad]$ and $x'\in [\Sigma,B\grad]$, we have : $V_{\C}(\Sigma\coprod \Sigma',T\cup T',x\cup x')=V_{\C}(\Sigma,T,x)\pt V_{\C}(\Sigma',T',x')$. It follows from the construction of the Turaev-Viro invariant that~: $p_{\Sigma\coprod \Sigma',T\cup T'}=p_{\Sigma,T}\pt_{\Bbbk}p_{\Sigma', T'}$ (\cite{Tu}). Therefore we obtain~: $p_{\Sigma\coprod\Sigma',T\cup T',x\coprod x'}=p_{\Sigma,T,x}\pt_{\Bbbk}p_{\Sigma', T',x'}$. It results that the functor $\m{H}_{\C}$ is monoidal for the objects. Let us show that $\m{H}_{\C}$ is monoidal for the morphisms. For every morphisms $(M,F) : (\Sigma_1,g_1)\fd (\Sigma'_1,g'_1)$ and $(M',F') : (\Sigma_2,g_2)\fd (\Sigma'_2,g'_2)$, we have :
 $$\widetilde{HTV_{\C}}(M\coprod M',F\cup F')_{(x_{g_1},x_{g'_1}),(x_{g_2},x_{g'_2})}=\widetilde{HTV_{\C}}(M,F)_{x_{g_1},x_{g'_1}}\pt \widetilde{HTV_{\C}}(M',F')_{x_{g_2},x_{g'_2}}\, ,$$
 where for every $i\in\{1,2\}$ $x_{g_i}$ (resp. $x_{g'_i}$) the homotopy class of $g_i$ (resp. $g'_i$). It follows that $\m{H}_{\C}$ is an HQFT.
\qed

For every spherical category $\C$, the HQFT $\m{H}_{\C}$ is called \emph{Turaev-Viro HQFT}.

\subsubsection{Group categories $\C_{G,\alpha}$ with $G$ an abelian group}

We will compute the vector spaces associated to a closed surface of genus $g$ by the Turaev-Viro HQFT in the case of group categories defined for an abelian group. From now on, we denote a closed surface of genus $g$ by $\Sigma_g$. Let $G$ be an abelian finite group and $\alpha\in H^3(G,\Bbbk^*)$, we denote by $\C_{G,\alpha}$ the associated group category.

Below is a singular triangulation $T_g$ of a closed surface of genus $g=2$.

\[
\xy
(0,0)*+{x}="A";(10,8)*+{x}="B";(20,10)*+{x}="C";(30,8)*+{x}="D";(40,0)*+{x}="E";(30,-8)*+{x}="G";(10,-8)*+{x}="I";(20,-10)*+{x}="H";
{\ar^{a_1} "A";"B"}; {\ar^{b_1} "B";"C"}; {\ar_{a_1} "D";"C"};
{\ar_{b_1} "E";"D"}; {\ar_{b_2}
"A";"I"};{\ar_{a_2} "I";"H"}; {\ar^{a_2} "E";"G"};{\ar^{b_2}
"G";"H"};{\ar_{c_1} "A";"C"};{\ar^{d_1} "E";"C"};{\ar_{e}
"A";"E"};{\ar^{d_2} "A";"H"};{\ar_{c_2} "E";"H"};
\endxy
\]

The above singular triangulation admits a unique 0-simplex and four 1-simplexes.  We can extend the construction to closed surfaces of genus $g\geq 2$. We denote by $a_1,...,a_g,b_1,...,b_g,c_1,...,c_g,d_1,...d_g,e$ the 1-simplexes of $T_g$.  Let us describe the set of colorings of $T_g$. Let $c$ be a coloring of the above singular triangulation of $\Sigma_g$. We set~:
\begin{align*}
c(a_i)&=h_i\, ,\\
c(b_i)&=k_i\, ,\\
\end{align*}
for every $1\leq i \leq g$. With the colors $h_i$ and $k_i$, we obtain the colors of each 1-simplex of the triangulation, indeed~: $c(c_i)=a_ib_i$ and $c(d_i)=b_ia_i$. The definition of colorings gives the following relations~: $[a_1,b_1]...[a_g,b_g]=1$, with $[a_i,b_i]=a_ib_ia_i^{-1}b_i^{-1}$ for every $i$. Since $G$ is abelian group, the previous  condition is always verified. Reciprocally the data $(h_1,...,h_g,k_1,...,k_g)$ defines a coloring $c$ of $\Sigma_g$, in such a way~: $c(a_i)=h_i$ and $c(b_i)=k_i$.

Let us describe the gauge action on the set of colorings of $T_g$. Let us recall that the singular triangulation $T_g$ admits a unique 0-simplex thus the gauge group of $T_g$ can be identified to $G$. Let $c=(h_1,...,h_g,k_1,...,k_g)$ and $c'=(h'_1,...,h'_g,k'_1,...,k'_g)$ be two colorings of $T_g$, the colorings $c$ and $c'$ are equivalent if and only if there exists $h\in G$ such that~:
\begin{align*}
h'_i&= hh_ih^{-1}=h_i\, ,\\
g'_i&= hg_ih^{-1}=g_i\, ,
\end{align*}
for every $i$. It follows that for every $x\in [\Sigma_g,BG]$ the set $Col_x(\Sigma_g)$ contains a unique coloring.

We will describe the vector spaces associated to $\Sigma_g$ by the Turaev-Viro HQFT $\m{H}_{\C_{G,\alpha}}$. First, we will describe the vector space $V_{\C_{G,\alpha}}(\Sigma_g,T_g)$ defined in the construction of the Turaev-Viro invariant. In the category $\C_{G,\alpha}$ the vector space associated to the oriented 2-simplexes is a one dimensional vector space, it follows~: $$
V_{\C_{G,\alpha}}(\Sigma_g,T_g)=\bigoplus_{c\in Col(T_g)}\Bbbk=\bigoplus_{(h_1,...,h_g,k_1,...,k_g)\in G^{2g}}\Bbbk\, .
$$
We have shown that for every homotopy class $x\in [\Sigma_g,BG]$ the set $Col_x(T_g)$ contains a unique coloring, it follows that~: $V_{\C_{G,\alpha}}(\Sigma_g,T_g,x)=\Bbbk$. In order to describe the vector space $\m{H}_{\C_{G,\alpha}}(\Sigma_g,x)$ for every $x\in [\Sigma_g,BG]$ we will determine the idempotent $p_{\Sigma_g,T_g,x}:V_{\C_{G,\alpha}}(\Sigma_g,T_g,x)\fd V_{\C_{G,\alpha}}(\Sigma_g,T_g,x)$ for every $x\in [\Sigma_g,BG]$. Let $x\in [\Sigma_g,BG]$, the linear map $p_{\Sigma_g,x}$ is an idempotent of a one dimensional vector space, thus $p_{\Sigma_g,x}=0$ or $p_{\Sigma_g,x}=\id$. Since the Turaev-Viro invariant of the manifold $\Sigma_g\times I$ with a fixed colorings of the boundary is equal to 1, it follows that $p_{\Sigma_g,x}=\id$. As a consequence we have~: $\m{H}_{\C_{G,\alpha}}(\Sigma,x)=\Bbbk$. The lemma \ref{lem:idemspli} gives the splitting of the Turaev-Viro TQFT $\m{V}_{\C_{G\alpha}}$ into blocks given by the Turaev-Viro HQFT $\m{H}_{\C_{G,\alpha}}$. It follows that in the case of group categories $\C_{G,\alpha}$ with $G$ an abelian group, the splitting of the Turaev-Viro TQFT by the Turaev-Viro HQFT is maximal.

\begin{pro}\label{pro:max}
  Let  $G$ an abelian group, $\alpha\in H^3(G,\Bbbk^*)$, $\C_{G,\alpha}$ be a group category, $g$ be a positive integer and $\Sigma_g$ be a closed surface of genus $g$, we have~:  $$
  \m{V}_{\C_{G,\alpha}}(\Sigma_g)=\bigoplus_{x\in [\Sigma_g,BG]}\m{H}_{\C_{G,\alpha}}(\Sigma_g,x)\,
  $$ with $\m{H}_{\C_{G,\alpha}}(\Sigma_g,x)=\Bbbk$ for every $x\in[\Sigma_g,BG]$.
\end{pro}

\subsubsection*{The torus $S^1\times S^1$}

We compute the Turaev-Viro HQFT of the torus $S^1\times S^1$ in the case of the quantum group $U_q(\frak{sl}_2)$ with $q$ a root of unity.

Below is a singular triangulation $T_{S^1\times S^1}$ of the torus $S^1\times S^1$~:
\[
\xy
(0,0)*+{x}="A"; (0,25)*+{x}="B";
(35,0)*+{x}="C";(35,25)*+{x}="D";
"A";"B" **\dir{-} ?(.5)*\dir{>}+(-2,0)*{\scriptstyle a}; "A";"C" **\dir{-} ?(.5)*\dir{>>}+(0,-2)*{\scriptstyle b} ; "B";"D" **\dir{-} ?(.5)*\dir{>>}+(0,2)*{\scriptstyle b}; "C";"D" **\dir{-} ?(.5)*\dir{>}+(2,0)*{\scriptstyle a}; "A";"D" **\dir{--} ?(.5)*\dir{>}+(2,-1)*{\scriptstyle e} ;
(20,-8)*{T_{S^1\times S^1}};
\endxy
\]

\bigskip

There are three oriented 1-simplexes $a$, $b$ and $e$ and one 0-simplex $x$. Let  $c$ be a coloring of this singular triangulation, then we obtain the triple $(c(a),c(b),c(e))$. By definition of the colorings, the triple $(c(a),c(b),c(e))$ is an admissible triple. Reciprocally every admissible triple $(i,j,k)$ defines a coloring $c$ of $T_{S^1\times S^1}$, in such a way~: $c(a)=i$, $c(b)=j$ and $c(e)=k$. Let us describe the gauge action on the set of colorings of $T_{S^1\times S^1}$. Let $c=(i,j,k)$ be a coloring of $T_{S^1\times S^1}$, we recall that $\cs{c}$ is the coloring $c$ with value in the graduator $\Gamma_{U_q(\frak{sl}_2)}=\Zz_2$ and $\cs{?}:\Lambda_{U_q(\frak{sl}_2)}\fd \Zz_2$ is the projection map. We will consider the group $\Zz_2$ with the multiplicative notation. By definition of the $\Zz_2$-coloring, we have the following relation~: $\cs{i}\cs{j}=\cs{k}$. It follows that the $\Zz_2$-coloring $\cs{c}$ is given by the pair $(\cs{i},\cs{j})\in \Zz_2\times \Zz_2$. The singular triangulation $T_{S^1\times S^1}$ admits a unique 0-simplex, thus the gauge group of $T_{S^1\times S^1}$ can be identified to $\Zz_2$. Let $c=(i,j,k)$ and $c'=(i',j',k')$ be two colorings of $T_{S^1\times S^1}$, they are equivalent if and only if there exists a gauge $h\in \Zz_2$ such that~:
\begin{align*}
\cs{i'}&=h\cs{i}h^{-1}=\cs{i}\, ,\\
\cs{j'}&=h\cs{j}h^{-1}=\cs{j}\, .
\end{align*}
It follows that the equivalence class of a coloring $c=(i,j,k)$ is given by the pair $(\cs{i},\cs{j})\in \Zz_2\times \Zz_2$. There exists four homotopy classes in $[S^1\times S^1,B\Zz_2]$ which corresponds to the equivalent classes of colorings~: $(1,1)$, $(1,-1)$, $(-1,1)$ and $(-1,-1)$. From now on, we denote the homotopy classes of $[S^1\times S^1,B\Zz_2]$ by the corresponding equivalent classes of colorings.

Let us describe the vector spaces associated to $S^1\times S^1$ by the Turaev-Viro HQFT. First, we will describe the vector space $V_{U_q(\frak{sl}_2)}(S^1\times S^1,T_{S^1\times S^1})$ defined in the construction of the Turaev-Viro and then after we will describe the idempotents. For the sake of clarity, the vector space $V_{U_q(\frak{sl}_2)}(S^1\times S^1,T_{S^1\times S^1})$ will be denoted $V(S^1\times S^1)$. In the case of $U_q(\frak{sl}_2)$, the vector space associated to an oriented 2-simplex is a one dimensional vector space, it follows~:
\begin{align*}
V(S^1\times S^1)&=\bigoplus_{c\in Col(T_{S^1\times S^1})}\Bbbk=\bigoplus_{(i,j,k)~admissible~triple }\Bbbk\, ,\\
V(S^1\times S^1,(1,1))&=\bigoplus_{c\in Col_{(1,1)}(T_{S^1\times S^1})}\Bbbk=\bigoplus_{\substack{(i,j,k)~admissible~triple\\ \cs{i}=\cs{j}=1}}\Bbbk\, ,\\
V(S^1\times S^1,(1,-1))&=\bigoplus_{c\in Col_{(1,-1)}(T_{S^1\times S^1})}\Bbbk=\bigoplus_{\substack{(i,j,k)~admissible\\ \cs{i}=1~  and ~ \cs{j}=-1}}\Bbbk\, ,\\
V(S^1\times S^1,(-1,1))&=\bigoplus_{c\in Col_{(-1,1)}(T_{S^1\times S^1})}\Bbbk=\bigoplus_{\substack{(i,j,k)~admissible\\ \cs{i}=-1~  and ~ \cs{j}=1}}\Bbbk\, ,\\
V(S^1\times S^1,(-1,-1))&=\bigoplus_{c\in Col_{(-1,-1)}(T_{S^1\times S^1})}\Bbbk=\bigoplus_{\substack{(i,j,k)~admissible\\ \cs{i}=-1~  and ~ \cs{j}=-1}}\Bbbk\, .
\end{align*}
 By symmetry of the admissible triple we have $V(S^1\times S^1,(1,-1))=V(S^1\times S^1,(-1,1))$. For every coloring $c=(i,j,k)\in Col_{(1,-1)}(T_{S^1\times S^1})$, the triple $(j,i,k)$ is an admissible triple. Thus the triple $c'=(j,i,k)$ is a coloring of $T_{S^1\times S^1}$ such that $\cs{c'}=(-1,1)$. It follows that every coloring $c=(i,j,k)\in Col_{(1,-1)}(T_{S^1\times S^1})$ defines a coloring $\tilde{c}=(j,i,k)\in Col_{(1,-1)}(T_{S^1\times S^1})$. We obtain the following bijection~:
\begin{align*}
Col_{(1,-1)}(T_{S^1\times S^1})&\fd Col_{(-1,1)}(T_{S^1\times S^1})\\
c=(i,j,k) & \ap \tilde{c}=(j,i,k)\, .
\end{align*}
Furthermore by symmetry of $S^1\times S^1\times I$, we have~:%
\begin{align}
HTV_{U_q(\frak{sl}_2)}(S^1\times S^1\
\times I, (1,-1,1))_{c,c'}&= HTV_{U_q(\frak{sl}_2)}(S^1\times S^1\times I, (-1,1,1))_{\tilde{c},\tilde{c'}}\, ,\label{idepotent1}\\
HTV_{U_q(\frak{sl}_2)}(S^1\times S^1\times I, (1,-1,-1))_{\tilde{c},\tilde{c'}}&= HTV_{U_q(\frak{sl}_2)}(S^1\times S^1\times I, (-1,1,-1))_{\tilde{c},\tilde{c'}}\, ,\label{idepotent2}\\
\end{align}

for every $c, c'\in Col_{(1,-1)}(T_{S^1\times S^1})$. Since the vector spaces $V(S^1\times S^1,(1,-1))$ and $V(S^1\times S^1,(-1,1))$ are the same, the above equalities (\ref{idepotent1}) and (\ref{idepotent2}) give~: $p_{S^1\times S^1,T_{S^1\times S^1},(1,-1)}=p_{S^1\times S^1,T_{S^1\times S^1},(-1,1)}$. It follows~:
$$
\m{H}_{U_q(\frak{sl}_2)}(S^1\times S^1,(1,-1))=\m{H}_{U_q(\frak{sl}_2)}(S^1\times S^1,(-1,1))\, .
$$

\subsubsection*{Computation}

For $r=3$, we have the following idempotents~:
\begin{align*}
p_{S^1\times S^1,(1,1)}&=(1)\, ,\\
p_{S^1\times S^1,(1,-1)}&=(1)\, ,\\
p_{S^1\times S^1,(-1,-1)}&=(1)\, ,
\end{align*}

and the dimensions of the vector spaces are~:
\begin{align*}
\dim(\m{H}_{U_q(\frak{sl}_2)}(S^1\times S^1,(1,1))) &= 1\, ,\\
\dim(\m{H}_{U_q(\frak{sl}_2)}(S^1\times S^1,(1,-1)))&= 1 \, ,\\
\dim(\m{H}_{U_q(\frak{sl}_2)}(S^1\times S^1,(-1,-1)))&=1\,.
\end{align*}

For $r=4$, we have the following idempotents~:
\begin{align*}
p_{S^1\times S^1,(1,1)}&=\frac{1}{4}\left(\begin{array}{cccc}3 & 1 &1 & 1\\ 1 &3 &-1 &-1 \\ 1 &-1 & 3 & -1\\ 1& -1 &-1 & 3\end{array}\right)\, ,\\
&\\
p_{S^1\times S^1,(1,-1)}&=\left(\begin{array}{cc}1 & 0 \\ 0 &1 \end{array}\right)\, ,\\
&\\
p_{S^1\times S^1,(-1,-1)}&=\left(\begin{array}{cc}1 & 0 \\ 0 &1 \end{array}\right)\, ,\\
\end{align*}

The dimensions of the vector spaces are~:
\begin{align*}
\dim(\m{H}_{U_q(\frak{sl}_2)}(S^1\times S^1,(1,1))) &= 3\, ,\\
\dim(\m{H}_{U_q(\frak{sl}_2)}(S^1\times S^1,(1,-1)))&= 2 \, ,\\
\dim(\m{H}_{U_q(\frak{sl}_2)}(S^1\times S^1,(-1,-1)))&=2\, ,
\end{align*}

For $r=5$, we set $A$ a primitive $10$th root of unity such that $A^2$ is a primitive $5$th root of unity. For the sake of clarity, we denote by $\sigma$ the sum $A+A^{-1}$. This notation is used in \cite{matveev}. We obtain the following idempotents~:
\begin{align*}
p_{S^1\times S^1,(1,1)}&= \frac{1}{2+\sigma}\left(\begin{array}{ccccc}2 & 1 & 1 & 1 &\sigma^{-3/2} \\ 1 & 3 & 1-\sigma & 1-\sigma & -\sigma^{-5/2}\\  1 &  1-\sigma& 3 & 1-\sigma & -\sigma^{-5/2}\\  1 &  1-\sigma & 1-\sigma & 3 & -\sigma^{-5/2}\\  \sigma^{-3/2} & -\sigma^{-5/2} & -\sigma^{-5/2} & -\sigma^{-5/2} & 4\sigma-3\\\end{array}\right)\, , \\
&\\
p_{S^1\times S^1,(1,1)}&= \frac{1}{2+\sigma}\left(\begin{array}{ccccc}2 & 1 & 1 & 1 &\sigma^{-3/2} \\ 1 & 3 & 1-\sigma & 1-\sigma & -\sigma^{-5/2}\\  1 &  1-\sigma& 3 & 1-\sigma & -\sigma^{-5/2}\\  1 &  1-\sigma & 1-\sigma & 3 & -\sigma^{-5/2}\\  \sigma^{-3/2} & -\sigma^{-5/2} & -\sigma^{-5/2} & -\sigma^{-5/2} & 4\sigma-3\\\end{array}\right)\, , \\
&\\
p_{S^1\times S^1,(-1,-1)}&=\frac{1}{2+\sigma}\left(\begin{array}{ccccc}2 & 1 & 1 & 1 &\sigma^{-3/2} \\ 1 & 3 & 1-\sigma & 1-\sigma & -\sigma^{-5/2}\\  1 &  1-\sigma& 3 & 1-\sigma & -\sigma^{-5/2}\\  1 &  1-\sigma & 1-\sigma & 3 & -\sigma^{-5/2}\\  \sigma^{-3/2} & -\sigma^{-5/2} & -\sigma^{-5/2} & -\sigma^{-5/2} & 4\sigma-3\\\end{array}\right)\, .
\end{align*}

The dimensions of the vector spaces are~:
\begin{align*}
\dim(\m{H}_{U_q(\frak{sl}_2)}(S^1\times S^1,(1,1))) &= 4\, ,\\
\dim(\m{H}_{U_q(\frak{sl}_2)}(S^1\times S^1,(1,-1)))&= 4 \, ,\\
\dim(\m{H}_{U_q(\frak{sl}_2)}(S^1\times S^1,(-1,-1)))&=4\, ,
\end{align*}

\medskip

For $r=6$,  we give one idempotent~:
\begin{align*}
p_{S^1\times S^1,(1,1)}&= \frac{1}{12}\left(\begin{array}{ccccccccccc}5 & 3 & 3 & 3 & 1& 1 & 1 & 1 & 1 &1 & \sqrt{2} \\ 3 & 9 & -1 & -1  & 3 & -1 & -1& 1 & 1& -1 &-\sqrt{2}\\ 3 & -1 & 9 & -1 &  -1& 3& -1 &-1 &1 & 1& -\sqrt{2}\\3 & -1 & -1  & 9 &-1 & -1& 3& 1& -1& 1& -\sqrt{2}\\  1 & 3  & -1  &-1  & 5& 1 & 1& -3& -3& 1&\sqrt{2} \\ 1 & -1 & 3 &-1 & 1&5 & 1 &1 &-3 & -3& \sqrt{2}\\ 1 & -1& -1 &3 & 1& 1&5 & -3 &1 &-3 &\sqrt{2} \\  1 & 1  & -1 &  1& -3&1 & -3 &9 & -1& -1&\sqrt{2} \\ 1 & 1 & 1& -1& -3 &-3 & 1& -1&9 &-1 & \sqrt{2}\\ 1 & -1 & 1& 1& 1&-3 &-3 &-1 &-1 &9 & \sqrt{2} \\ \sqrt{2} & -\sqrt{2}& -\sqrt{2}& -\sqrt{2}&\sqrt{2} & \sqrt{2}&\sqrt{2} &\sqrt{2} &\sqrt{2} &\sqrt{2} & 10 \\\end{array}\right)\, , \\
& \\%
\end{align*}

The dimensions of the vector spaces are~:
\begin{align*}
\dim(\m{H}_{U_q(\frak{sl}_2)}(S^1\times S^1,(1,1))) &= 7\, ,\\
\dim(\m{H}_{U_q(\frak{sl}_2)}(S^1\times S^1,(1,-1)))&=  6\, ,\\
\dim(\m{H}_{U_q(\frak{sl}_2)}(S^1\times S^1,(-1,-1)))&=6\, ,
\end{align*}

\section{The twisted homological Turaev-Viro invariant}\label{sec:twist}

We recall the construction of the homological twisted generalized Turaev-Viro invariant \cite{homyetter}.

Let $\C$ be a semisimple tensor category with braiding and $M$ be a closed 3-manifold. We denote $H_1(M,\aut)$ the first homology group of $M$ with coefficients in $\aut$. Let $h\in H_1(M,\aut)$ and $\alpha$ be a representative of $h$, $\alpha^e\in \aut$ is the coefficient of $e$ in $\aut$. For every scalar object $X$, $\alpha^e(X)=\alpha^e_X \id_X$, with $\alpha^e_X\in \Bbbk^*$. Let $M$ be a closed 3-manifold and $h\in H_1(M,\aut)$, the homological twisted Turaev-Viro invariant of $(M,h)$ is the scalar~:
\begin{equation}
Y_{\C}(M,h)=\dc^{-n_0(T)}\sum_{c\in Col(T)}\prod_{e\in T^1}\alpha^e_{c(e)}w_cW_c\, ,\label{YTV}
\end{equation}
with $n_0(T)$ the number of 0-simplexes of a triangulation $T$. The scalar $Y_{\C}(M,h)$ does not depend on the choice of the triangulation of $M$ and the representative of $h$. In \cite{homyetter}, Yetter prove that this scalar is an invariant for a semisimple tensor category with braiding. With some changes in the proof we can show that the invariant is well defined for spherical categories.
\begin{pro}
  Let $\C$ be s spherical category, $M$ be a closed 3-manifold and $h\in H_1(M,\aut)$. The scalar $\dis{Y_{\C}(M,h)=\dc^{-n_0(T)}\sum_{c\in Col(T)}\prod_{e\in T^1}\alpha^e_{c(e)}w_cW_c}$, with $\alpha$ a representative of $h$, is an invariant of the pair $(M,h)$.
\end{pro}
\pf
Let us show that for every spherical category $\C$, $Y_{\C}$ does not depend on the choice of a representative of h. First, the scalar $\dis{\prod_{e\in T^1}\alpha^e_{c(e)}}$ does not depend on the choice of the orientation of the 1-simplexes. Indeed for every 1-simplex $e$, we have : $\alpha^{\ov{e}}=(\alpha^e)^{-1}$, with $\ov{e}$ the 1-simplex $e$ endowed with the opposite orientation. According to the proposition \ref{pro:gradaut}, for every oriented 1-simplex $(01)$, the monoidal automorphism $\alpha^{(01)}$ is determined by a group morphism $\epsilon^{(01)}\in Hom(\grad,\Bbbk^*)$. The morphism $\epsilon^{(01)}$  verifies the relation : $\epsilon^{(01)}=(\epsilon^{(10)})^{-1}=\epsilon^{(10)}$.

Let $\alpha$ and $\alpha'$ be two representative of $h\in H_1(M,\aut)$, there exists a 2-chain $\dis{\beta=\sum_{f\in T^2} \beta^ff}$, such that $\alpha'=\alpha\delta(\beta)$, with $\delta$ the boundary operator. If $\beta=\beta^{(012)}(012)$ then we have~:
\begin{align*}
\alpha'^{(01)}&=\alpha^{(01)}\beta^{(012)}\, ,\\
\alpha'^{(12)}&=\alpha^{(12)}\beta^{(012)}\, ,\\
\alpha'^{(02)}&=\alpha^{(02)}\beta^{(012)}\, ,\\
\alpha'^{e}&=\alpha^{e} \qquad \mbox{if $e$ is not a subsimplex of $(012)$} \, .
\end{align*}
Let $\epsilon^{(012)}\in Hom(\grad,\Bbbk^*)$ be the group morphism which defined the monoidal automorphism $\beta^{(012)}$.

If $\alpha'=\alpha\delta\beta$, with $\dis{\beta=\sum_{f\in T^2}\beta^ff}$, since for every scalar object $X$, $\beta_X^{(012)}=(\beta_X^{(012)})^{-1}=\epsilon^{(012)}(X)\id_X$, it follows for every coloring $c\in Col(T)$~:
\begin{align*}
\prod_{e\in T^1}\alpha'^e_{c(e)} &= \prod_{e\in T^1}\alpha^e_{c(e)} \prod_{e\in T^1}\prod_{\substack{f\\ e<f}}\epsilon^f(\cs{c(e)})\\
&= \prod_{e\in T^1}\alpha^e_{c(e)} \prod_{f\in T^2}\epsilon^f(\cs{\hat{f}_1})\epsilon^f(\cs{\hat{f}_2})\epsilon^f(\cs{\hat{f}_3}),
\end{align*}
where $\hat{f}_i$ is the 1-simplex obtained from $f$ by removing the 0-simplex $i$. We set $f=(012)$, we have : $\epsilon^f(\cs{\hat{f}_1})\epsilon^f(\cs{\hat{f}_2})\epsilon^f(\cs{\hat{f}_3})=\epsilon^f(\cs{c(12)})\epsilon^f(\cs{c(20)})\epsilon^f(\cs{c(01)})=1$, since $\epsilon^f\in Hom(\grad,\Bbbk^*)$. Thus for a spherical category $\C$, $Y_{\C}(M,h)$ does not depend on the choice of a representative of $h$. The proof of the independence on the choice of the triangulation is the same as the proof of Yetter \cite{homyetter}. We replace the cycle by a trivial cycle since the region removed and replaced by the Pachner is contractible. Finally, the proofs of invariance under the Pachner moves is the same as the proof of the Turaev-Viro invariance under the Pachner moves.
\qed

%
\begin{theorem}\label{yetterHQFT}
Let $\C$ be a spherical category, $M$ be 3-manifold.  For every $h\in H_1(M,\aut)$, we have~:
$$
Y_{\C}(M,h)=\sum_{x\in[M:B\grad]}(h:x)HTV_{\C}(M,x)\, ,
$$
with $\dis{(h:x)=\prod_{e\in T^1}\alpha^e_{c(e)}}$, $\alpha$ a representative of $h$ and $c\in Col_x(T)$.
\end{theorem}
\pf
Let $h\in H_1(M,\aut)$ and $\alpha$ be a representative of $h$. For every coloring $c_0$ of $T_0$, we have~:
$$
Y_{\C}(M,x,c_0)=\left(\dc\right)^{-n_0(T)}\sum_{x\in [M:B\grad]}\sum_{c\in Col_x(T)}\prod_{e\in T^1}\alpha_{c(e)}^ew_cW_c\, .
$$
According to the proposition \ref{pro:gradaut}, for every 1-simplex $e$ there exists a unique group morphism $\epsilon^e$ from $\grad$ to $\Bbbk^*$ such that for every scalar object $X$ : $\alpha_X^e=\epsilon^e(\cs{X})$. Since $\dis{\alpha=\sum_e\alpha^ee}$ is a 1-chain, we have : $\dis{\delta \alpha=\sum_{i\in T^0}\sum_{e\in T^1}(\alpha^e)^{\pm 1}i=0}$, the sign is given by the following rule : $\delta(01)=0-1$. By definition of $Col_x(T)$, for every colorings $c,c'\in Col_x(T)$ there exists a gauge $\delta : T^0\fd \grad$ such that $|c'|=|c|^{\delta}$, it follows~:
\begin{equation*}
\prod_{e\in T^1}\alpha^e_{c'(e)}=\prod_{e\in T^1}\alpha^e_{c(e)}\prod_{i\in T^0}\prod_{e\in T^1,i\in e}\epsilon^e(\delta(i))^{-1}\, ,
\end{equation*}
the sign is given by the following relation : $|c|^{\delta}(01)=\delta(0)|c|(01)\delta^{-1}(1)$. Since $\alpha$ is a 1-chain, we obtain : $\prod_{e\in T^1}\alpha^e_{c(e)}=\prod_{e\in T^1}\alpha^e_{c'(e)}$. Thus this product only depends on the homotopy class of $x$ and does not depend on the choice of a representative of $h$. We set : $\dis{(\alpha:x)=\prod_{e\in T^1}\alpha^e_{c(e)}}$, for every coloring $c\in Col_x(T)$ and it follows : $\dis{Y_{\C}(M,\alpha)=\sum_{x\in [M,B\grad]}(h:x)HTV_{\C}(M,x)}$.
\qed

\section{Tables}\label{sec:table}
\subsection{Group categories}
\begin{center}
\begin{tabularx}{16cm}{|c|c|c|X|X|}
\hline
N & $\alpha$& Manifold & TV invariant & HTV invariant \\
\hline
\multirow{2}*{2} & \multirow{2}*{$\exp(2i\pi/4)$} & $L(2p+1,q)$ & 1/2 & 1/2 \\ \cline{3-5}
 & & $L(2p,q)$ & $1/2(1+(-1)^p)$ & $(1/2,(-1)^p/2)$ \\ \cline{3-5}
\hline
\multirow{6}*{3}& \multirow{6}*{$\exp(2i\pi/9)$} & $L(3p+2,q)$ & 1/3 & 1/3 \\ \cline{3-5}
& & $L(3p+1,q)$ & 1/3 & 1/3 \\ \cline{3-5}
&  & \multirow{2}*{$L(3p,q)$} & $1/3(1+2\exp(2i\pi p/3))$ &$(1/3,1/3\exp(2i\pi p/3),$ $1/3\exp(2i\pi p/3))$  \mbox{if $n = 1$ mod $3$}\\ \cline{4-5}
& & & $1/3(1+2\exp(4i\pi p/3))$ &$(1/3,1/3\exp(4i\pi p/3),$ $1/3\exp(4i\pi p/3))$ \mbox{if $n = 2$ mod $3$}\\\cline{4-5}
\hline
\end{tabularx}
\end{center}

\begin{center}
\begin{tabularx}{16cm}{|c|c|c|X|X|}
\hline
\multirow{5}*{4} & \multirow{5}*{$\exp(2i\pi/16)$} & $L(4p+3,q)$ & 1/4 & 1/4\\ \cline{3-5}
& & $L(4p+2,q)$ & $1/4(1+\exp(i\pi p/2))$ & $(1/4,1/4\exp(i\pi p/2))$  \\ \cline{3-5}
& & $L(4p+1,q)$ &1/4 & 1/4\\ \cline{3-5}
& & \multirow{2}*{$L(4p,q)$} & $1/4(1+3\exp(i\pi p/2))$ & $(1/4,1/4\exp(i\pi p/2),$ $1/4\exp(i\pi p/2),1/4\exp(i\pi p/2))$ if $n=1$ mod $4$ \\ \cline{4-5}
& &  & $1/4(1+exp(i\pi p /2)+2\exp(-i\pi p/2))$ & $(1/4,1/4\exp(-i\pi p/2),$ $1/4\exp(i\pi p/2),1/4\exp(-i\pi p/2))$ if $n=3$ mod $4$ \\ \cline{4-5}
\hline
\multirow{6}*{5} & \multirow{6}*{$\exp(2i\pi/25)$} & $L(5p+4,q)$ & 1/5 & 1/5 \\ \cline{3-5}
& & $L(5p+3,q)$ & 1/5 & 1/5 \\ \cline{3-5}
& & $L(5p+2,q)$ & 1/5 & 1/5 \\ \cline{3-5}
& & $L(5p+1,q)$ & 1/5 & 1/5 \\ \cline{3-5}
& & \multirow{2}*{$L(5p,q)$} & $1/5(1+2\exp(2i\pi p/5)+2\exp(-2i\pi p/5))$ & $(1/5,\exp(2i\pi p/5),\exp(-2i\pi p/5),$ $\exp(-2i\pi p/5),\exp(2i\pi p/5))$\\ \cline{4-5}
& & & $1/5(1+2\exp(4i\pi p/5)+2\exp(-4i\pi p/5))$ & $(1/5,\exp(-4i\pi p/5),\exp(4i\pi p/5),$ $\exp(4i\pi p/5), \exp(-4i\pi p/5))$ if $n=3$ mod 5\\ \cline{4-5}
\hline
\multirow{8}*{6} & \multirow{8}*{$\exp(2i\pi/36)$} & $L(6p+5,q)$ & 1/6 & 1/6\\ \cline{3-5}
& & $L(6p+4,q)$ & 1/6 & 1/6 \\ \cline{3-5}
& & \multirow{2}*{$L(6p+3,q)$} & $1/6(1+2\exp(4i\pi (2p+1)/6))$& $(1/6,1/6\exp(4i\pi (2p+1)/6),1/6\exp(4i\pi (2p+1)/6))$ if $n=1$ or $4$ mod 6\\ \cline{4-5}
& &  & $1/6(1+2\exp(-4i\pi (2p+1)/6))$& $(1/6,1/6\exp(-4i\pi (2p+1)/6),1/6\exp(-4i\pi (2p+1)/6))$ if $n=2$ or $5$ mod 6\\ \cline{3-5}
& & $L(6p+2,q)$ & $1/6(1+(-1)^{p+1})$&$(1/6,(-1)^{p+1}/6)$ \\ \cline{3-5}
& & $L(6p+1,q)$ & 1/6 & 1/6 \\ \cline{3-5}
& & \multirow{2}*{$L(6p,q)$} &$1/6(1+(-1)^p+2\exp(2i\pi p/6)+2\exp(-4i\pi p/6))$ & $(1/6,1/6\exp(2i\pi p/6),$ $ 1/6\exp(-4i\pi p/6),$ $(-1)^p/6,1/6\exp(-4i\pi p/6),$ $1/6\exp(2i\pi p/6))$ if $n=1$ mod 6\\ \cline{4-5}
& &  &$1/6(1+(-1)^p+2\exp(-2i\pi p/6)+2\exp(4i\pi p/6)$ & $(1/6,1/6\exp(-2i\pi p/6),$ $1/6\exp(4i\pi p/6),$ $(-1)^p/6,1/6\exp(4i\pi p/6),$ $1/6\exp(-2i\pi p/6))$ if $n=5$ mod 6\\ \cline{3-5}
\hline
\multirow{7}*{7} & \multirow{7}*{$\exp(2i\pi/49)$} & $L(7p+6,q)$ & 1/7 & 1/7\\ \cline{3-5}
& & $L(7p+5,q)$ & 1/7& 1/7 \\ \cline{3-5}
& & $L(7p+4,q)$ & 1/7& 1/7\\ \cline{3-5}
& & $L(7p+3,q)$ & 1/7& 1/7\\ \cline{3-5}
& & $L(7p+2,q)$ &1/7 &1/7 \\ \cline{3-5}
& & $L(7p+1,q)$ &1/7 & 1/7\\ \cline{3-5}
& & \multirow{2}*{$L(7p,q)$} & $1/7(1+2\exp(2i\pi p/7)+2\exp(4i\pi p/7)+2\exp(-6i\pi p/7))$& $(1/7,1/7\exp(2i\pi p/7),$ $1/7\exp(-6i\pi p/7),1/7\exp(4i\pi p/7),$ $1/7\exp(4i\pi p/7),1/7\exp(-6i\pi p/7),$ $1/7\exp(2i\pi p/7))$ if $n=1,2$ or $4$ mod 7 \\ \cline{4-5}
& &  & $1/7(1+2\exp(-2i\pi p/7)+2\exp(-4i\pi p/7)+2\exp(6i\pi p/7))$& $(1/7,1/7\exp(-2i\pi p/7),$ $1/7\exp(-4i\pi p/7),1/7\exp(6i\pi p/7),$ $1/7\exp(6i\pi p/7),1/7\exp(-4i\pi p/7),$ $1/7\exp(-2i\pi p/7))$ if $n=3,5$ or $6$ mod 7\\
\hline
\end{tabularx}
\end{center}

\begin{center}
\begin{tabularx}{16cm}{|c|c|c|X|X|}
\hline
N & $\alpha$ & Manifold & TV invariant & HTV invariant \\
\hline
\multirow{6}*{8} & \multirow{8}*{$\exp(2i\pi/64)$} & $L(8p+7,q)$ & 1/8 & 1/8 \\ \cline{3-5}
& & $L(8p+6,q)$ & 1/8 & 1/8 \\ \cline{3-5}
& & $L(8p+5,q)$ & 1/8 & 1/8 \\ \cline{3-5}
& & $L(8p+4,q)$ &$(-1)^p/8$ &$(1/8,-1/8,(-1)^p/8)$ \\ \cline{3-5}
& & $L(8p+3,q)$ & 1/8 & 1/8 \\ \cline{3-5}
& & $L(8p+2,q)$ & 0 & $(1/8,-1/8)$ \\ \cline{3-5}
& & $L(8p+1,q)$ & 1/8 & 1/8 \\ \cline{3-5}
& & \multirow{3}*{$L(8p,q)$} &$ 1/4((-1)^p+2\exp(i\pi p/4))$ & $(1/8,1/8\exp(i\pi /4), (-1)^p/8,$ $1/8\exp(i\pi p/4),-1/8,1/8\exp(i\pi p/4),$ $(-1)^p/8,1/8\exp(i\pi p/4)$ if $n=1$ mod 8\\ \cline{4-5}
& & &$ 1/4((-1)^p+2\exp(3i\pi p/4))$ & $(1/8,1/8\exp(3i\pi /4), (-1)^p/8,$ $1/8\exp(3i\pi p/4),-1/8,1/8\exp(3i\pi p/4),$ $(-1)^p/8,1/8\exp(3i\pi p/4)$ if $n=3$ mod 8\\ \cline{4-5}
& & &$ 1/4((-1)^p+2\exp(-3i\pi p/4))$ & $(1/8,1/8\exp(-3i\pi /4), (-1)^p/8,$ $1/8\exp(-3i\pi p/4),-1/8,$ $1/8\exp(-3i\pi p/4),(-1)^p/8,$ $1/8\exp(-3i\pi p/4)$ if $n=5$ mod 8\\ \cline{3-5}
\hline
\multirow{12}*{9} & \multirow{9}*{$\exp(2i\pi/81)$} & $L(9p+8,q)$ & 1/9 & 1/9 \\ \cline{3-5}
& & $L(9p+7,q)$ & 1/9 & 1/9 \\ \cline{3-5}
& & $L(9p+6,q)$ & 1/9 & 1/9 \\ \cline{3-5}
& & $L(9p+5,q)$ & 1/9 & 1/9 \\ \cline{3-5}
& & $L(9p+4,q)$ & 1/9 & 1/9 \\ \cline{3-5}
& & $L(9p+3,q)$ & $3/9$ & $(1/9,1/9,1/9)$\\ \cline{3-5}
& & $L(9p+2,q)$ & 1/9 & 1/9 \\ \cline{3-5}
& & $L(9p+1,q)$ & 1/9 & 1/9 \\ \cline{3-5}
& & \multirow{4}*{$L(9p,q)$} & $1/9(3+2\exp(2i\pi p/9)+2\exp(8i\pi p/9)+\exp(4i\pi p/9)+\exp(-4i\pi p/9))$ & $(1/9,1/9\exp(2i\pi p/9),1/9\exp(8i\pi p/9),$ $1/9\exp(4i\pi p/9),1/9\exp(-4i\pi p/9),$ $1/9,1/9\exp(2i\pi p/9),1/9\exp(8i\pi p/9),$ $1/9)$ if $n=1$ mod 9\\ \cline{4-5}
& & & $1/9(3+2\exp(4i\pi p/9)+2\exp(-4i\pi p/9)+\exp(8i\pi p/9)+\exp(-8i\pi p/9))$ & $(1/9,1/9\exp(4i\pi p/9),$ $ 1/9\exp(-4i\pi p/9),1/9\exp(8i\pi p/9),$ $1/9\exp(-8i\pi p/9),1/9,$ $1/9\exp(4i\pi p/9),1/9\exp(-4i\pi p/9),$ $1/9)$ if $n=2$ mod 9\\ \cline{4-5}
& & & $1/9(3+2\exp(8i\pi p/9)+2\exp(2i\pi p/9)+2\exp(-4i\pi p/9))$ & $(1/9,1/9\exp(8i\pi p/9),$ $1/9\exp(2i\pi p/9),1/9\exp(-4i\pi p/9),$ $1/9\exp(8i\pi p/9),1/9,1/9\exp(2i\pi p/9),$ $1/9\exp(-4i\pi p/9),$ $1/9)$ if $n=4$ or $7$ mod 9\\ \cline{4-5}
& & & $1/9(3+2\exp(4i\pi p/9)+2\exp(-2i\pi p/9)+2\exp(-8i\pi p/9))$ & $(1/9,1/9\exp(4i\pi p/9),$ $1/9\exp(-2i\pi p/9),1/9\exp(-8i\pi p/9),$ $1/9,1/9\exp(4i\pi p/9),1/9\exp(-2i\pi p/9),$ $1/9\exp(-8i\pi p/9),1/9)$ if $n=5$ or 8 mod 9\\ \cline{4-5}
\hline
\end{tabularx}
\end{center}

\bibliographystyle{amsplain}
\bibliography{E:/these/biblio/biblioordre}

\providecommand{\bysame}{\leavevmode\hbox to3em{\hrulefill}\thinspace}
\providecommand{\MR}{\relax\ifhmode\unskip\space\fi MR }
\providecommand{\MRhref}[2]{%
  \href{http://www.ams.org/mathscinet-getitem?mr=#1}{#2}
}
\providecommand{\href}[2]{#2}
\begin{thebibliography}{10}

\bibitem{BW}
J.W. Barrett and B.W. Westbury, \emph{{Invariants of piecewise-linear
  3-manifolds}}, Trans. Amer. Math. Soc. \textbf{348} (1996), no.~10,
  3997--4022.

\bibitem{BHMV}
C.~Blanchet, N.~Habegger, G.~Masbaum, and P.~Vogel, \emph{{Topological quantum
  field theories derived from the Kauffman bracket}}, {Topology} \textbf{34}
  (1995), no.~4, 883--927.

\bibitem{6j}
J.~S. Carter, D.~E. Flath, and M.~Saito, \emph{The classical and quantum
  6j-symbols}, vol. Mathematical Notes, Princeton University Press, 1995.

\bibitem{FK}
J.~Fr\"ohlich and T.~Kerler, \emph{{Quantum groups, quantum categories and
  quantum field theories}}, {Lecture Notes in Mathematics}, vol. 1542,
  {S}pringer-{V}erlag, {B}erlin, {N}ew-{Y}ork, 1993.

\bibitem{GK}
I.~Gelfand and D.~Kazhdan, \emph{{Invariant of three dimensional manifolds}},
  Geometric and Functional Analysis \textbf{6} (1996), no.~2, 268--300.

\bibitem{Jacorubin}
W.~Jaco and J.~H. Rubinstein, \emph{0-efficient triangulations of 3-manifolds},
  J. differential geometry \textbf{65} (2003), 61--168.

\bibitem{Kassel}
C.~Kassel, \emph{{Q}uantum groups}, {Graduate text in Mathematics}, vol. 155,
  {S}pringer-{V}erlag, {B}erlin, {N}ew-{Y}ork, 1995.

\bibitem{TL}
T.~Le and V.~Turaev, \emph{Quantum groups and ribbon g-categories}, J. Pure
  Appl. Algebra \textbf{178} (2003), no.~2, 169--185.

\bibitem{matveev}
S.~Matveev, \emph{Algorithmic topology and classification of 3-manifolds},
  second ed., Algorithms and Computation in Mathematics, no.~9, Springer,
  Berlin, 2007.

\bibitem{Moore}
G.~Moore and N.~Seiberg, \emph{{Classical and quantum field theory}}, {Comm.
  Math Phy.} \textbf{123} (1989), 177--254.

\bibitem{Pachner}
U.~Pachner, \emph{{P.L. homeomorphic manifolds are equivalent by elementary
  shellings}}, {European journal of Combinatorics} (1991), 129--145.

\bibitem{TVDW}
J.~Petit, \emph{The turaev-viro invariant associated to group categories},
  preprint (2009).

\bibitem{Roberts}
J.~Roberts, \emph{{Skein theory and Turaev-Viro invariants}}, {Topology}
  \textbf{34} (1995), 771--787.

\bibitem{THQFT}
V.~Turaev, \emph{Homotopy field theory in dimension 3 and crossed
  group-categories}, arXiv:math/0005291, 2000.

\bibitem{Tu}
V.G. Turaev, \emph{{Quantum invariants of knots and 3-manifolds}}, {Walter de
  Gruyter}, 1994.

\bibitem{TV}
V.G. Turaev and O.~Viro, \emph{{State sum invariants of 3-manifolds and quantum
  6j-symbol}}, {Topology} \textbf{31} (1992), 865--902.

\bibitem{homyetter}
D.N. Yetter, \emph{{Homologically Twisted Invariants Related to (2+1)- and
  (3+1)-Dimensional State-Sum Topological Quantum Field Theories}},
  arXiv:hep-th/9311082 (1993).

\end{thebibliography}
\end{document}